\documentclass{article}
\PassOptionsToPackage{numbers}{natbib}

\usepackage[final]{neurips_2019}

\usepackage[utf8]{inputenc} 
\usepackage[T1]{fontenc}    
\usepackage{hyperref}       
\usepackage{url}            
\usepackage{booktabs}       
\usepackage{amsfonts}       
\usepackage{nicefrac}       
\usepackage{microtype}      
\usepackage{amsmath,amssymb,amsthm}
\usepackage{enumerate}
\usepackage{algorithmic,algorithm}
\usepackage{enumitem}
\usepackage{graphicx}
\usepackage{xcolor}

\newcommand\Fcal{{\mathcal F}}

\newcommand\E{\mb{E}}

\newcommand\mb[1]{\mathbb{ #1 }}

\newcommand\xk{x_k}

\newcommand{\argmin}{\operatornamewithlimits{argmin}}

\newcommand\Real{{\mathbb R}}
\newcommand\defin{:=}
\newcommand\mtd{{\mathcal M}}

\newcommand\vs{\vspace*{-0.15cm}}

\newtheorem{theorem}{Theorem}
\newtheorem{proposition}[theorem]{Proposition}
\newtheorem{lemma}[theorem]{Lemma}
\newtheorem{definition}[theorem]{Definition}
\newtheorem{corollary}[theorem]{Corollary}

\newcommand\kmone{{k\text{--}1}}

\newcommand\kmtwo{{k\text{--}2}}
\newcommand\kmthree{{k\text{--}3}}
\newcommand\jmone{{j\text{--}1}}

\newcommand\eg{{\it e.g.}}

\graphicspath{{./figures/}}

\newenvironment{itemize*}%
{\begin{itemize}%
\vspace*{-0.1cm}
   \setlength{\itemsep}{0pt}%
      \setlength{\parskip}{0pt}%
      \setlength{\topskip}{0pt}%
      \setlength{\parsep}{0pt}%
      \setlength{\partopsep}{0pt}%
      \setlength{\itemindent}{0pt}%
      \setlength{\topsep}{0pt}}%
{\end{itemize}
\vspace*{-0.1cm}
}

\title{A Generic Acceleration Framework \\ for Stochastic Composite Optimization}

\author{%
   Andrei Kulunchakov and Julien Mairal \\
      Univ. Grenoble Alpes, Inria, CNRS, Grenoble INP, LJK, 38000 Grenoble, France\\
      \texttt{andrei.kulunchakov@inria.fr and julien.mairal@inria.fr} \\
}

\begin{document}

\maketitle

\begin{abstract}
In this paper, we introduce various mechanisms to obtain accelerated first-order stochastic
optimization algorithms when the objective function is convex or strongly
convex. Specifically, we extend the Catalyst approach originally designed for
deterministic objectives to the stochastic setting. Given an optimization
method with mild convergence guarantees for strongly convex problems,
the challenge is to accelerate convergence to a noise-dominated region, and
then achieve convergence with an optimal worst-case complexity depending on the
noise variance of the gradients.
A side contribution of our work is also a generic analysis that can
handle inexact proximal operators, providing new insights about the robustness of 
stochastic algorithms when the proximal operator cannot be exactly computed.
 
\end{abstract}

\section{Introduction}\label{sec:intro}
In this paper, we consider stochastic composite optimization problems of the form
\begin{equation}
   \min_{x \in \Real^p} \left\{  F(x) \defin f(x)  + \psi(x)  \right\}~~~~\text{with}~~~~ f(x) = \E_{\xi}[ \tilde{f}(x,\xi)], \label{eq:risk}
\end{equation}
where the function $f$ is convex, or $\mu$-strongly convex, and $L$-smooth (meaning differentiable
with $L$-Lipschitz continuous gradient), and $\psi$ is a possibly non-smooth convex  lower-semicontinuous
function. For instance, $\psi$ may be the 
$\ell_1$-norm, which is known to induce sparsity, 
or an indicator function of a convex set~\cite{hiriart_urruty_lemarechal_1993ii}.
The random variable~$\xi$ corresponds to data samples. When the amount of training data
is finite, the expectation $\E_{\xi}[ \tilde{f}(x,\xi)]$ can be replaced by a
finite sum, a setting that has attracted a lot of attention in machine learning
recently, see, \eg, 
\cite{saga,defazio2014finito,gower2018stochastic,konevcny2017semi,miso,sarah,proxsvrg} for incremental algorithms and~\cite{accsvrg,kovalev2019don,conjugategradient,catalyst_jmlr,accsdca,zhou2018direct,zhou2018simple} for accelerated variants.

Yet, as noted in \cite{bottou2008tradeoffs}, one is typically not interested in
the minimization of the empirical risk---that is, a finite sum of functions---with
high precision, but instead, one should focus on the expected risk involving the
true (unknown) data distribution. When one can draw an infinite number of
samples from this distribution, the true risk~(\ref{eq:risk}) may be minimized
by using appropriate stochastic optimization techniques. Unfortunately, 
fast methods designed for deterministic objectives would not apply
to this setting; methods based on stochastic approximations
admit indeed optimal ``slow'' rates that are typically $O(1/\sqrt{k})$ for convex functions and $O(1/k)$
for strongly convex ones, depending on the exact assumptions made on the problem, where $k$ is the number of noisy gradient evaluations~\cite{nemirovski}.

Better understanding the gap between deterministic and stochastic optimization
is one goal of this paper. Specifically, we are interested in Nesterov's acceleration
of gradient-based approaches~\cite{nesterov1983,nesterov}. In a nutshell, gradient descent
or its proximal variant applied to a $\mu$-strongly convex $L$-smooth function achieves an
exponential convergence rate $O((1-\mu/L)^k)$ in the worst case in function values, and a sublinear
rate $O(L/k)$ if the function is simply convex ($\mu=0$). By interleaving the algorithm with clever
extrapolation steps, Nesterov showed that faster convergence could be achieved,
and the previous convergence rates become $O((1-\sqrt{\mu/L})^k)$ and
$O(L/k^2)$, respectively.
Whereas no clear geometrical intuition seems to appear in the literature to
explain why acceleration occurs, proof techniques to show accelerated
convergence~\cite{fista,nesterov,tseng} and extensions to a large class of other
gradient-based algorithms are now well
established~\cite{accsvrg,chambolle2015remark,catalyst_jmlr,nesterov2012,accsdca}.

Yet, the effect of Nesterov's acceleration to stochastic objectives remains
poorly understood since existing unaccelerated algorithms such as stochastic
mirror descent~\cite{nemirovski} and their variants already achieve the optimal
asymptotic rate.  Besides, negative results also exist, showing that Nesterov's
method may be unstable when the gradients are computed
approximately~\cite{d2008smooth,inexactnesterov}.  Nevertheless, several
approaches such as
\cite{aybat2019universally,cohen2018acceleration,devolder2011stochastic,ghadimi2012optimal,ghadimi2013optimal,kwok2009,kulunchakov2019estimate,Lan2012,xiao2010dual}
have managed to show that acceleration may be useful to forget faster the
algorithm's initialization and reach a region dominated by the noise of
stochastic gradients; then, ``good'' methods are expected to asymptotically converge with a
rate exhibiting an optimal dependency in the noise variance~\cite{nemirovski},
but with no dependency on the initialization.  A major challenge  is then to achieve the optimal rate for these two
regimes.

In this paper, we consider an optimization method $\mtd$ with the following
property: given an auxiliary strongly convex objective
function $h$, we assume that $\mtd$ is able to produce iterates $(z_t)_{t \geq0}$ with expected linear convergence to a noise-dominated region---that is, such that
\begin{equation}
   \E[h(z_t) - h^\star] \leq C (1-\tau)^t (h(z_0)-h^\star) + B \sigma^2, \label{eq:aux}
\end{equation}
where $C, \tau, B  > 0$, $h^\star$ is the minimum
function value, and $\sigma^2$ is an upper bound on the variance of stochastic 
gradients accessed by~$\mtd$, which we assume to be uniformly bounded. Whereas such an assumption  
has limitations, it remains the most standard  one
for stochastic optimization (see~\cite{bottou2018optimization,nguyen2018sgd} for more realistic
settings in the smooth case).
The class of methods satisfying~(\ref{eq:aux}) is relatively large.  For
instance, when $h$ is $L$-smooth, the stochastic gradient descent method (SGD) with
constant step size $1/L$ and iterate averaging satisfies~(\ref{eq:aux}) with $\tau=\mu/L$, $B=1/L$, and $C=1$, see~\cite{kulunchakov2019estimate}.

\paragraph{Main contribution.}
In this paper, we extend the Catalyst
approach~\cite{catalyst_jmlr} to general stochastic problems.\footnote{All objectives addressed by the original Catalyst approach are deterministic, even though they may be large finite sums. Here, we consider general expectations as defined in~(\ref{eq:risk}).} Under mild conditions,
our approach is able to turn~$\mtd$ into a converging algorithm 
with a worst-case expected complexity that decomposes into two parts: the first
one exhibits an accelerated convergence rate in the sense of Nesterov and shows
how fast one forgets the initial point; the second one corresponds to the
stochastic regime and typically depends (optimally in many cases) on~$\sigma^2$.
Note that even though we only make assumptions about the behavior of~$\mtd$ on strongly convex sub-problems~(\ref{eq:aux}), we also treat the case where the objective~(\ref{eq:risk}) is convex, but not strongly convex.

To illustrate the versatility of our approach, we consider the stochastic finite-sum
problem \cite{smiso,hofmann_variance_2015,lan2018random,zheng2018lightweight},
where the objective~(\ref{eq:risk}) decomposes into $n$ components
$\tilde{f}(x,\xi) = \frac{1}{n}\sum_{i=1}^n \tilde{f}_i(x,\xi)$ and $\xi$ is  
a stochastic perturbation, coming, \eg, from data augmentation or noise
injected during training to improve generalization or privacy
(see~\cite{kulunchakov2019estimate,miso}). The underlying finite-sum structure may also result from clustering
assumptions on the data~\cite{hofmann_variance_2015}, or from distributed computing~\cite{lan2018random}, a setting beyond the scope of our paper.
Whereas it was shown in~\cite{kulunchakov2019estimate} that classical variance-reduced stochastic optimization methods such as SVRG~\cite{proxsvrg},
SDCA~\cite{accsdca}, SAGA~\cite{saga}, or MISO~\cite{miso}, can be made robust to noise,
the analysis of~\cite{kulunchakov2019estimate} is only able to accelerate the SVRG approach. 
With our acceleration technique, all of the aforementioned methods can be modified such that they find a point~$\hat{x}$ satisfying $\E[F(\hat{x})-F^\star] \leq \varepsilon$ with global iteration complexity, for the $\mu$-strongly convex case,
\begin{equation}
   \tilde{O}\left( \left(n + \sqrt{n\frac{L}{\mu}}\right)\log\left(\frac{F(x_0)-F^\star}{\varepsilon} \right)   +  \frac{{\sigma}^2}{\mu\varepsilon} \right).  \label{eq:cplx_finitesum}
\end{equation}
The term on the left is the optimal complexity for finite-sum optimization~\cite{accsvrg,tightbound_yossi}, up to logarithmic terms in $L,\mu$ hidden in the $\tilde{O}(.)$ notation, and the term on the right is the optimal complexity for $\mu$-strongly convex stochastic objectives~\cite{ghadimi2012optimal} where $\sigma^2$ is due to the perturbations $\xi$. As Catalyst~\cite{catalyst_jmlr}, the price to pay compared to non-generic direct acceleration techniques~\cite{accsvrg,kulunchakov2019estimate} is a logarithmic factor.

\paragraph{Other contributions.}
In this paper, we generalize the analysis of
Catalyst~\cite{catalyst_jmlr,paquette2017catalyst} to handle various new cases.
Beyond the ability to deal with stochastic optimization problems, our approach (i) improves 
Catalyst by allowing sub-problems of the form~(\ref{eq:aux}) to be solved approximately \emph{in expectation}, which is more realistic than the deterministic requirement made in~\cite{catalyst_jmlr} and which is also critical for stochastic optimization, (ii)
leads to a new accelerated stochastic gradient descent algorithms for composite optimization with similar guarantees as~\cite{ghadimi2012optimal,ghadimi2013optimal,kulunchakov2019estimate},
(iii) handles the analysis of accelerated proximal gradient descent methods
with inexact computation of proximal operators, improving the results
of~\cite{schmidt2011convergence} while also treating the stochastic setting.

Finally, we note that the extension of Catalyst we propose is easy to
implement.  The original Catalyst method introduced in~\cite{catalyst_nips}
indeed required solving a sequence of sub-problems while controlling carefully
the convergence, \eg, with duality gaps. For this reason, Catalyst has sometimes been
seen as theoretically appealing but not practical enough~\cite{scieur2017nonlinear}.
Here, we focus on a simpler and more practical variant presented later
in~\cite{catalyst_jmlr}, which consists of solving sub-problems with a fixed
computational budget, thus removing the need to define stopping criterions for sub-problems.
The code used for our experiments is available here: \url{http://github.com/KuluAndrej/NIPS-2019-code}.

\section{Related Work on Inexact and Stochastic Proximal Point Methods.}
Catalyst is based on the inexact accelerated proximal point
algorithm~\cite{newppa}, which consists in solving approximately a sequence of sub-problems and updating two sequences $(x_k)_{k \geq 0}$ and $(y_k)_{k \geq 0}$ by
\begin{equation}
    x_k \approx \argmin_{x \in \Real^p} \left\{ h_k(x) \defin F(x) + \frac{\kappa}{2}\|x - y_\kmone\|^2 \right\} ~~~~\text{and}~~~~ y_k = x_k + \beta_k(x_k-x_\kmone), \label{eq:ppa}
\end{equation}
where~$\beta_k$ in $(0,1)$ is obtained from Nesterov's acceleration
principles~\cite{nesterov},~$\kappa$ is a well chosen regularization parameter, and~$\|\cdot\|^2$ is the Euclidean norm. The method~$\mtd$ is used to obtain an approximate
minimizer of $h_k$; when $\mtd$
converges linearly, it may be shown that the resulting algorithm~(\ref{eq:ppa})
enjoys a better worst-case complexity than if $\mtd$ was used directly on $f$, see~\cite{catalyst_jmlr}.

Since asymptotic linear convergence is out of reach when $f$ is a stochastic
objective, a classical strategy consists in replacing $F(x)$ in~(\ref{eq:ppa})
by a finite-sum approximation obtained by random sampling, leading to
deterministic sub-problems. Typically without Nesterov's acceleration (with $y_k=x_k$),
this strategy is often called the stochastic proximal point
method~\cite{asi2018stochastic,bertsekas2011incremental,kulis2010implicit,toulis2015stable,toulis2016towards}.
The point of view we adopt in this paper is different and is based on the
minimization of surrogate functions $h_k$ related to~(\ref{eq:ppa}), but which 
are more general and may take other forms than $F(x)+\frac{\kappa}{2}\|x-y_\kmone\|^2$.

\section{Preliminaries: Basic Multi-Stage Schemes}\label{sec:restart}
In this section, we present two simple multi-stage mechanisms to improve
the worst-case complexities of stochastic optimization methods, before introducing acceleration principles.

\vs
\paragraph{Basic restart with mini-batching or decaying step sizes.}
Consider an optimization method $\mtd$ with convergence rate~(\ref{eq:aux})
and assume that there exists a hyper-parameter to control a trade-off between the bias $B\sigma^2$ and 
the computational complexity.
Specifically, we assume that the bias can be reduced by an arbitrary factor
$\eta < 1$, while paying a factor $1/\eta$ in
terms of complexity per iteration (or $\tau$ may be reduced by a factor $\eta$, thus slowing down convergence).
This may occur in two cases:
\begin{itemize*}
   \item by using a mini-batch of size $1/\eta$ to sample gradients, which replaces $\sigma^2$ by $\eta \sigma^2$;
   \item or the method uses a step size proportional to $\eta$ that can be chosen arbitrarily small.
 \end{itemize*}
      For instance, stochastic gradient descent with constant step size and iterate averaging is compatible with both scenarios~\cite{kulunchakov2019estimate}.  Then, 
      consider a target accuracy $\varepsilon$ and define
      the sequences $\eta_k = 1/2^k$ and 
      $\varepsilon_k = 2B \sigma^2 \eta_k$ for $k \geq 0$.
      We may now solve successively the problem up to accuracy~$\varepsilon_k$---\eg, with a constant number $O(1/\tau)$ steps of $\mtd$ when using mini-batches of size $1/\eta_k=2^k$ to reduce the bias---and by using the solution of iteration $\kmone$ as a warm restart.
      As shown in Appendix~\ref{appendix:restart}, the scheme converges and
      the worst-case complexity to achieve the accuracy $\varepsilon$ in expectation is 
         \begin{equation}
            O\left( \frac{1}{\tau} \log\left( \frac{C (F(x_0)-F^\star)}{\varepsilon}  \right)+ \frac{B \sigma^2 \log(2C)}{\tau \varepsilon}   \right).\label{eq:restart2}
            \end{equation}
            For instance, one may run SGD with constant step size $\eta_k/L$ at stage $k$ with iterate averaging as in~\cite{kulunchakov2019estimate}, which yields $B=1/L$, $C=1$, and $\tau = \mu/L$. Then, the left term is the classical complexity $O((L/\mu)\log(1/\varepsilon))$ of the (unaccelerated) gradient descent algorithm for deterministic objectives, whereas the right term is the optimal complexity for stochastic optimization in $O(\sigma^2/\mu\varepsilon)$. Similar restart principles appear for instance in~\cite{aybat2019universally} in the design of a multistage accelerated SGD algorithm.

\vs
\paragraph{Restart: from sub-linear to linear rate with strong convexity.}
A natural question is whether asking for a linear rate in~(\ref{eq:aux}) for
strongly convex problems is a strong requirement. Here, we
show that a sublinear rate is in fact sufficient for our needs by 
generalizing a restart technique introduced in~\cite{ghadimi2013optimal} for stochastic
optimization, which  was previously used for
deterministic objectives in~\cite{iouditski2014primal}.

Specifically, consider an optimization method~$\mtd$ such that the convergence rate~(\ref{eq:aux}) is replaced by 
 \begin{equation}
   \E[h(z_t) - h^\star] \leq \frac{D \|z_0-z^\star\|^2}{2t^d} + \frac{B \sigma^2}{2}, \label{eq:aux2}
\end{equation}
where $D, d > 0$ and $z^\star$ is a minimizer of~$h$.
Assume now that $h$ is $\mu$-strongly convex with $D \geq \mu$ and 
consider restarting $s$ times the method~$\mtd$, each time running $\mtd$ for constant $t' = \lceil (2D/\mu)^{1/d} \rceil$ iterations.
Then, it may be shown (see Appendix~\ref{appendix:restart}) that the
relation~(\ref{eq:aux}) holds with constant $t=s t'$, $\tau = \frac{1}{2t'}$, and $C=1$.
If a mini-batch or step size mechanism is available, we may then proceed as
before and obtain a converging scheme with
complexity~(\ref{eq:restart2}), \eg, by using mini-batches of exponentially increasing sizes once the method reaches a noise-dominated region, and by using a restart frequency 
of order $O(1/\tau)$.

\section{Generic Multi-Stage Approaches with Acceleration}
\newcommand\HypA{($\mathcal H_1$)}
\newcommand\HypB{($\mathcal H_2$)}
\newcommand\HypC{($\mathcal H_3$)}
\newcommand\HypCb{($\mathcal H_4$)}
\newcommand\hypA{$\mathcal H_1$}
\newcommand\hypB{$\mathcal H_2$}
\newcommand\hypC{$\mathcal H_3$}
\newcommand\hypCb{($\mathcal H_4$)}
We are now in shape to introduce a generic acceleration framework that
generalizes~(\ref{eq:ppa}). Specifically, given some
point~$y_\kmone$ at iteration~$k$, we consider
a surrogate function $h_k$ related to a parameter $\kappa >0$, an approximation error $\delta_k \geq 0$, and an optimization method~$\mtd$ that satisfy the following properties:
\begin{itemize}
\addtolength{\leftmargin}{-0.5cm}
\item[\HypA] $h_k$ is $(\kappa+\mu)$-strongly convex, where $\mu$ is the strong convexity parameter of $f$; 
\item[\HypB] $\E[h_k(x) | {\mathcal F}_{\kmone}] \leq F(x) + \frac{\kappa}{2}\|x-y_{\kmone}\|^2$ for $x=\alpha_\kmone x^\star+(1-\alpha_\kmone)x_\kmone$, which is deteministic given the past information $\Fcal_\kmone$ up to iteration $\kmone$ and $\alpha_\kmone$ is given in Alg.~\ref{alg:meta};
\item[\HypC] $\mtd$ can provide the exact minimizer~$x_k^\star$ of $h_k$ and a point $x_k$ (possibly equal to $x_k^\star$) such that $\E[F(x_k)] \leq
\E[h_k^\star] + \delta_k$ where $h_k^\star = \min_{x} h_k(x)$.
\end{itemize}
The generic acceleration
framework is presented in Algorithm~\ref{alg:meta}.
Note that the conditions on $h_k$ bear similarities with estimate sequences
introduced by Nesterov~\cite{nesterov}; indeed, \HypC~is a direct generalization of (2.2.2) from~\cite{nesterov} and \HypB~resembles (2.2.1).  However, the choices of $h_k$
and the proof technique are significantly different, as we will see with
various examples below.
We also assume at the moment that the exact minimizer~$x_k^\star$ of $h_k$ is available, which differs from 
the Catalyst framework~\cite{catalyst_jmlr}; 
the case with approximate minimization  will 
be presented in Section~\ref{sec:variantB}.

\begin{algorithm}[hbtp]
\caption{Generic Acceleration Framework with Exact Minimization of $h_k$} \label{alg:meta}
\begin{algorithmic}[1]
\STATE {\bfseries Input:} $x_0$ (initial estimate); $\mtd$ (optimization method); $\mu$ (strong convexity constant); $\kappa$ (parameter for $h_k$); $K$ (number of iterations); $(\delta_k)_{k \geq 0}$ (approximation errors);
\STATE {\bfseries Initialization:} $y_0=x_0$; $q=\frac{\mu}{\mu+\kappa}$; $\alpha_0=1$ if $\mu=0$ or $\alpha_0 = \sqrt{q}$ if $\mu \neq 0$;
\FOR{ $k=1,\ldots,K$}
\STATE Consider a surrogate $h_k$ satisfying \HypA, \HypB ~and obtain $x_k, x_k^\star$ using $\mtd$ satisfying \HypC;
\STATE Compute $\alpha_k$ in $(0,1)$ by solving the equation
$\alpha_k^2 = (1-\alpha_k)\alpha_\kmone^2 + q \alpha_k$.
\STATE Update the extrapolated sequence
\begin{equation}
y_k = x_k^\star + \beta_k(x_k^\star - x_\kmone) + \frac{(\kappa+\mu)(1-\alpha_k)}{\kappa}(x_k - x_k^\star) ~~~\text{with}~~~ \beta_k = \frac{\alpha_\kmone (1-\alpha_\kmone)}{\alpha_\kmone^2 + \alpha_k}. \label{eq:variantA}
\end{equation}
\ENDFOR
\STATE {\bfseries Output:} $\xk$ (final estimate).
\end{algorithmic}
\end{algorithm}

\begin{proposition}[Convergence analysis for Algorithm~\ref{alg:meta}]\label{prop:A}
Consider Algorithm~\ref{alg:meta}. Then, 
         \begin{equation}
         \E[F(x_k)-F^\star] \leq \left\{  
         \begin{array}{ll}
         (1-\sqrt{q})^k\left(2(F(x_0)-F^\star) + \sum_{j=1}^k (1-\sqrt{q})^{-j} \delta_j\right)  & \text{if}~\mu \neq 0 \\
            \frac{2}{(k+1)^2}\left(\kappa\|x_0-x^\star\|^2 + \sum_{j=1}^k \delta_j(j+1)^2\right)  & \text{otherwise} \\
            \end{array}
            \right..   \label{eq:propAb}
            \end{equation}
            \end{proposition}
            The proof of the proposition is given in Appendix~\ref{appendix:proofs} and is based on an extension of the analysis of Catalyst~\cite{catalyst_jmlr}.
            Next, we present various application cases leading to algorithms with acceleration. 

            \vs
            \paragraph{Accelerated proximal gradient method.}
            When $f$ is deterministic and the proximal operator of $\psi$ (see Appendix~\ref{appendix:useful} for the definition) can be computed in closed form, choose $\kappa = L-\mu$ and define
            \begin{equation}
            h_k(x) \defin f(y_\kmone) + \nabla f(y_\kmone)^\top (x-y_\kmone) + \frac{L}{2}\|x-y_\kmone\|^2 + \psi(x). \label{eq:prox}
            \end{equation}
            Consider $\mtd$ that minimizes $h_k$ in closed form: $x_k\!=\!x_k^\star\!=\!\text{Prox}_{\psi/L}\left[y_\kmone - \frac{1}{L}\nabla f(y_\kmone)\right]$.
            Then, \HypA~is obvious; \HypB~holds from the convexity of $f$, and
            \HypC~with $\delta_k=0$ follows from classical inequalities for $L$-smooth
            functions~\cite{nesterov}. Finally, we recover 
            accelerated convergence rates~\cite{fista,nesterov}. 

            \vs \paragraph{Accelerated proximal point algorithm.} We consider $h_k$ given
            in~(\ref{eq:ppa}) with exact minimization (thus an unrealistic
                  setting, but conceptually interesting) with $\kappa=L-\mu$. Then, the assumptions~\HypA, \HypB, and
            \HypC~are satisfied with $\delta_k=0$ and we recover the accelerated rates of~\cite{newppa}.

            \vs \paragraph{Accelerated stochastic gradient descent with prox.} 
            A more
            interesting choice of surrogate is 
            \begin{equation} 
            h_k(x) \defin f(y_\kmone) + g_k^\top (x-y_\kmone) +
            \frac{\kappa+\mu}{2}\|x-y_\kmone\|^2 + \psi(x),
            \label{eq:prox} 
            \end{equation} 
            where $\kappa \geq L-\mu$ and 
            $g_k$ is an unbiased estimate of $\nabla
            f(y_\kmone)$---that is, $\E[g_k|\Fcal_\kmone]=\nabla
            f(y_\kmone)$---with variance bounded by~$\sigma^2$, following classical
            assumptions from the stochastic optimization
            literature~\cite{ghadimi2012optimal,ghadimi2013optimal,kwok2009}.  Then,
            \HypA~and \HypB~are satisfied given that~$f$ is convex.
            To characterize \HypC, consider
            $\mtd$ that minimizes $h_k$ in closed form:
            $x_k\!=\!x_k^\star\!=\!\text{Prox}_{{\psi}/{(\kappa+\mu)}}[y_\kmone - \frac{1}{\kappa+\mu}g_k]$,
            and define 
            $u_\kmone \defin \text{Prox}_{{\psi}/{(\kappa+\mu)}}[y_\kmone - \frac{1}{\kappa+\mu} \nabla
            f(y_\kmone)]$, which is deterministic given $\Fcal_\kmone$. Then,
            from~(\ref{eq:prox}), 
            \begin{displaymath} 
            \begin{split} 
            F(x_k) & \leq  h_k(x_k) + (\nabla f(y_\kmone)-g_k)^\top (x_k-y_\kmone) \qquad
            \qquad \text{(from $L$-smoothness of $f$)}\\ 
            & = h_k^\star + (\nabla f(y_\kmone)-g_k)^\top (x_k-u_\kmone) + (\nabla
                  f(y_\kmone)-g_k)^\top (u_\kmone-y_\kmone).
            \end{split} 
            \end{displaymath} 
            When taking expectations, the last term on the
            right disappears since $\E[g_k|{\mathcal F}_\kmone] = \nabla f(y_{\kmone})$:
            \begin{equation} 
            \begin{split} 
            \E[F(x_k)]  
            & \leq \E[h_k^\star] + \E[\|g_k-\nabla f(y_\kmone)\| \|x_k-u_\kmone\|] \\ 
            & \leq \E[h_k^\star] + \frac{1}{\kappa+\mu} \E\left[\|g_k-\nabla f(y_\kmone)\|^2\right] \leq
            \E[h_k^\star] + \frac{\sigma^2}{\kappa+\mu}, 
            \end{split} \label{eq:sgd2}
            \end{equation} 
            where we used
            the non-expansiveness of the proximal operator~\cite{moreau1965}.
            Therefore, \HypC~holds with $\delta_k\!=\!\sigma^2/(\kappa+\mu)$. The resulting algorithm
            is similar to~\cite{kulunchakov2019estimate} and offers the same guarantees. The novelty  of our approach
            is then a unified convergence proof for the deterministic and stochastic cases.

            \begin{corollary}[Complexity of proximal stochastic gradient algorithm, $\mu > 0$]\label{corollary:sgd}
            Consider Algorithm~\ref{alg:meta} with $h_k$ defined in~(\ref{eq:prox}).
            When $f$ is $\mu$-strongly convex, choose $\kappa=L-\mu$. Then,
            \begin{equation*} 
            \E[F(x_k)-F^\star] \leq \left( 1- \sqrt{\frac{\mu}{L}} \right)^k (F(x_0)-F^\star)  + \frac{\sigma^2}{\sqrt{\mu L}},
            \end{equation*} 
            \end{corollary}
            which is of the form~(\ref{eq:aux}) with $\tau=\sqrt{\mu/L}$ and $B=\sigma^2/(\sqrt{\mu L})$.
            Interestingly, the
            optimal complexity
            $O\left(\sqrt{L/\mu}\log((F(x_0)-F^\star)/\varepsilon)+\sigma^2/\mu\varepsilon\right)$ can be obtained by using the first restart
            strategy presented in Section~\ref{sec:restart}, see Eq.~(\ref{eq:restart2}), either by using increasing mini-batches or decreasing step sizes.

            When the objective is convex, but not strongly
            convex, Proposition~\ref{prop:A} gives a bias term $O(\sigma^2 k / \kappa)$ that increases linearly with $k$. Yet, the following corollary exhibits an optimal rate with finite horizon,
            when both $\sigma^2$ and an upper-bound on $\|x_0-x^\star\|^2$ are available.
            Even though non-practical, the result shows that our analysis
            recovers the optimal dependency in the noise level,
            as~\cite{ghadimi2013optimal,kulunchakov2019estimate} and others.
            \begin{corollary}[Complexity of proximal stochastic gradient algorithm,
            $\mu=0$]\label{cor:sgd2}  Consider a fixed budget $K$ of iterations of Algorithm~\ref{alg:meta} with $h_k$ defined in~(\ref{eq:prox}).  When
            $\kappa=\max(L,\sigma
                  (K+1)^{3/2}/\|x_0-x^\star\|)$, 
            \begin{equation*} 
            \E[F(x_K)-F^\star] \leq
            \frac{2L\|x_0-x^\star\|^2}{(K+1)^2} + \frac{3\sigma
               \|x_0-x^\star\|}{\sqrt{K+1}}.  
               \end{equation*} 
               \end{corollary}
               While all the previous examples use the choice $x_k=x_k^\star$, we will see in
               Section~\ref{subsec:finitesum} cases where we may choose $x_k \neq
               x_k^\star$. Before that, we introduce a variant
               when $x_k^\star$ is not available.

In principle, it is possible to design other surrogates, which would lead to new algorithms coming with convergence guarantees given by Propositions~\ref{prop:A} and~\ref{prop:B}, but the given examples~(\ref{eq:ppa}), (\ref{eq:prox}), and (\ref{eq:prox})  already cover all important cases considered in the paper for functions of the form~(\ref{eq:risk}).

               \subsection{Variant with Inexact Minimization} \label{sec:variantB}
               In this variant, presented in Algorithm~\ref{alg:metaB}, $x_k^\star$ is not available and we assume that $\mtd$ also satisfies:
               \begin{itemize}
               \vs
               \addtolength{\leftmargin}{-0.5cm}
               \item[\HypCb]  given $\varepsilon_k \geq 0$, $\mtd$ can provide a point $x_k$ such that $\E[h_k(x_k)-h_k^\star] \leq \varepsilon_k$.
               \end{itemize}
               \begin{algorithm}[hbtp!]
               \caption{Generic Acceleration Framework with Inexact Minimization of $h_k$} \label{alg:metaB}
               \begin{algorithmic}[1]
               \STATE {\bfseries Input:} same as Algorithm~\ref{alg:metaB};
\STATE {\bfseries Initialization:} $y_0=x_0$; $q=\frac{\mu}{\mu+\kappa}$; $\alpha_0=1$ if $\mu=0$ or $\alpha_0 = \sqrt{q}$ if $\mu \neq 0$;
\FOR{ $k=1,\ldots,K$}
\STATE Consider a surrogate $h_k$ satisfying \HypA, \HypB ~and obtain $x_k$ satisfying~\HypCb;
\STATE Compute $\alpha_k$ in $(0,1)$ by solving the equation
$\alpha_k^2 = (1-\alpha_k)\alpha_\kmone^2 + q \alpha_k$.
\STATE Update the extrapolated sequence $y_k = x_k + \beta_k(x_k-x_\kmone)$ with $\beta_k$ defined in~(\ref{eq:variantA});
\ENDFOR
\STATE {\bfseries Output:} $\xk$ (final estimate).
\end{algorithmic}
\end{algorithm}

The next proposition, proven in Appendix~\ref{appendix:proofs}, gives us some insight on how to achieve acceleration.
\begin{proposition}[Convergence analysis for
Algorithm~\ref{alg:metaB}]\label{prop:B} 
Consider Alg.~\ref{alg:metaB}. Then, for any $\gamma \in (0,1]$,
      \begin{equation*}
      \E[F(x_k)-F^\star] \leq \left\{  
      \begin{array}{ll}
      \left(1-\frac{\sqrt{q}}{2}\right)^k\!\!\left(2(F(x_0)-F^\star) + {4}\sum_{j=1}^k \left(1-\frac{\sqrt{q}}{2}\right)^{-j} \!\!\left( \delta_j + \frac{\varepsilon_j}{\sqrt{q}}\right) \right)  & \text{if}~\mu \neq 0 \\
      \frac{2e^{1+\gamma}}{(k+1)^2} \left( {\kappa}\|x_0-x^\star\|^2  +  \sum_{j=1}^k (j+1)^2\delta_j + \frac{(j+1)^{3+\gamma}\varepsilon_j}{\gamma}\right) & \text{if}~\mu=0. \\
      \end{array}
      \right.   
      \end{equation*}
      \end{proposition} 
      To maintain the accelerated rate, the sequence $(\delta_k)_{k \geq 0}$ needs to converge at a similar speed as in Proposition~\ref{prop:A}, but
      the dependency in~$\varepsilon_k$ is slightly worse.
      Specifically, when~$\mu$ is positive,  we may have both $(\varepsilon_k)_{k
      \geq 0}$ and $(\delta_k)_{k \geq 0}$ decreasing at a rate $O((1-\rho)^k)$ with $\rho < \sqrt{q}/2$, but we pay a factor $(1/\sqrt{q})$ compared to~(\ref{eq:propAb}).
      When $\mu=0$, the accelerated $O(1/k^2)$ rate is preserved whenever
      $\varepsilon_k=O(1/k^{4+2\gamma})$ and $\delta_k=O(1/k^{3+\gamma})$, but we pay a factor $O(1/\gamma)$ compared to~(\ref{eq:propAb}).

      \vs \paragraph{Catalyst~\cite{catalyst_jmlr}.} 
      When using $h_k$ defined in~(\ref{eq:ppa}), we recover the convergence rates of~\cite{catalyst_jmlr}. 
      In such a case $\delta_k=\varepsilon_k$ since $\E[F(x_k)] \leq
      \E[h_k(x_k)] \leq \E[h_k^\star]+\delta_k$.
In order to analyze the complexity of minimizing each~$h_k$ with $\mtd$ and
derive the global complexity of the multi-stage
algorithm, the next proposition, proven in Appendix~\ref{appendix:proofs}, 
   characterizes the quality of the initialization $x_\kmone$.

   \begin{proposition}[Warm restart for Catalyst]\label{prop:restart}
   Consider Alg.~\ref{alg:metaB} with $h_k$ defined in~(\ref{eq:ppa}). Then, for $k \geq 2$,
   \begin{equation}
   \E[h_k(x_\kmone)-h_k^\star] \leq \frac{3\varepsilon_\kmone}{2} + 54\kappa\max\left( \|x_\kmone-x^\star\|^2,\|x_\kmtwo-x^\star\|^2,\|x_{\kmthree}-x^\star\|^2\right),\label{eq:init}
   \end{equation}
   \end{proposition}
   where $x_{\text{--}1} \!=\! x_0$.
   Following~\cite{catalyst_jmlr}, we may now analyze the global complexity.
   For instance, when~$f$ is $\mu$-strongly convex, we may choose $\varepsilon_k=O((1-\rho)^k(F(x_0)-F^\star))$ with~$\rho = \sqrt{q}/3$. Then, 
   it is possible to show that Proposition~(\ref{prop:B}) yields $\E[F(x_k)-F^\star] = O(\varepsilon_k/{q})$ and from the inequality
   $\frac{\mu}{2}\|x_k-x^\star\|^2 \leq F(x_k)- F^\star$ and~(\ref{eq:init}), we have $\E[h_k(x_\kmone)-h_k^\star] = O(\frac{\kappa}{\mu q}\varepsilon_{\kmone}) =O(\varepsilon_\kmone/q^2)$.
   Consider now a method~$\mtd$ that behaves as~(\ref{eq:aux}). When $\sigma=0$, $x_k$ can be obtained in $O(\log(1/q)/\tau)=\tilde{O}(1/\tau)$ iterations of~$\mtd$ after initializing with $x_\kmone$. This allows us to obtain the global complexity $\tilde{O}((1/\tau\sqrt{q})\log(1/\varepsilon))$.
   For example, when $\mtd$ is the proximal gradient descent method, $\kappa=L$  and $\tau = ({\mu+\kappa})/({L+\kappa})$ yield the global complexity $\tilde{O}(\sqrt{L/\mu}\log(1/\varepsilon))$ of an accelerated method.

   Our results improve upon Catalyst~\cite{catalyst_jmlr} in two aspects that are crucial for stochastic optimization: (i) we allow the sub-problems to be solved in expectation, whereas Catalyst requires the stronger condition $h_k(x_k)-h_k^\star \leq \varepsilon_k$; (ii) Proposition~\ref{prop:restart} removes the requirement of~\cite{catalyst_jmlr} to perform a full gradient step for initializing the method~$\mtd$ in the composite case (see Prop. 12 in~\cite{catalyst_jmlr}).

   \vs
   \paragraph{Proximal gradient descent with inexact prox~\cite{schmidt2011convergence}.}
   The surrogate~(\ref{eq:prox}) with inexact minimization can be treated in the same
   way as Catalyst, which provides a unified proof for both problems.
   Then, we recover the results of~\cite{schmidt2011convergence},
   while allowing inexact minimization to be performed in expectation.

   \vs
   \paragraph{Stochastic Catalyst.}
   With Proposition~\ref{prop:restart}, we are in shape to consider stochastic problems
   when using a method~$\mtd$ that converges
   linearly as~(\ref{eq:aux}) with $\sigma^2 \neq 0$ for minimizing $h_k$.
   As in Section~\ref{sec:restart}, we also assume
   that there exists a mini-batch/step-size parameter $\eta$ that can reduce the bias by a
factor~$\eta < 1$ while paying a factor $1/\eta$ in terms of inner-loop complexity.
As above, we discuss the strongly-convex case and choose the same sequence
$(\varepsilon_k)_{k \geq 0}$.  In order to minimize $h_k$ up to accuracy $\varepsilon_k$, we set $\eta_k=\min(1,\varepsilon_k/(2B\sigma^2))$ such that $\eta_k B \sigma^2 \leq
\varepsilon_k/2$.
Then, the complexity to minimize $h_k$ with $\mtd$ when using the
initialization $x_\kmone$ becomes $\tilde{O}(1/\tau\eta_k)$, leading to the global complexity
\begin{equation}
   \tilde{O}\left( \frac{1}{\tau\sqrt{q}}\log\left(\frac{F(x_0)-F^\star}{\varepsilon}\right) + \frac{{B\sigma^2}}{{q^{3/2} \tau}\varepsilon}\right).  \label{eq:stoch_catalyst}
\end{equation}
Details about the derivation are given in Appendix~\ref{appendix:restart}. The
left term corresponds to the Catalyst accelerated rate, but it may be shown
that the term on the right is sub-optimal. Indeed, consider 
$\mtd$ to be ISTA with $\kappa=L-\mu$. Then, $B=1/L$, $\tau=O(1)$,
 and the right term becomes $\tilde{O}((\sqrt{L/\mu}) {\sigma^2}/{\mu\varepsilon})$, which is sub-optimal by a factor $\sqrt{{L}/{\mu}}$. Whereas this result is a negative one, suggesting that Catalyst is not robust to noise, we show in Section~\ref{subsec:finitesum} how to circumvent this for a large class of algorithms.

\vs
\paragraph{Accelerated stochastic proximal gradient descent with inexact prox.}
Finally, consider~$h_k$ defined in~(\ref{eq:prox}) but the proximal operator is computed approximately,
which, to our knowledge, has never been analyzed in the stochastic context. Then, it may be shown (see Appendix~\ref{appendix:restart} for details) that, even though~$\xk^\star$ is not available, Proposition~\ref{prop:B} holds nonetheless with
$\delta_k = 2 \varepsilon_k + {3\sigma^2}/({2(\kappa+\mu))}$.
Then, an interesting question is how small should $\varepsilon_k$ be to guarantee the optimal dependency with respect to~$\sigma^2$ as in Corollary~\ref{corollary:sgd}.
In the strongly-convex case, Proposition~\ref{prop:B} simply gives $\varepsilon_k = O(\sqrt{q}\sigma^2/(\kappa+\mu))$ such that $\delta_k \approx \varepsilon_k/\sqrt{q}$.

\subsection{Exploiting methods $\mtd$ providing strongly convex surrogates}\label{subsec:finitesum}
Among various application cases, we have seen an
extension of Catalyst to stochastic problems. To achieve
convergence, the strategy requires a mechanism to reduce the bias $B\sigma^2$ in~(\ref{eq:aux}),
\eg, by using mini-batches or decreasing step sizes.
Yet, the approach suffers from two issues: (i) some of the parameters are
based on unknown quantities such as~$\sigma^2$; (ii) the worst-case complexity
exhibits a sub-optimal dependency in $\sigma^2$, typically of order~$1/\sqrt{q}$ when $\mu > 0$.
Whereas practical workarounds for the first point are discussed in Section~\ref{sec:exp}, we now show 
how to solve the second one in some cases, by using Algorithm~\ref{alg:meta}
with an optimization method~$\mtd$, which is able not only to minimize an auxiliary
objective $H_k$, but also at the same time is able to provide a model $h_k$, typically a quadratic function, which is easy to minimize.
Consider then a method~$\mtd$ satisfying~(\ref{eq:aux}) and which produces, after $T$ steps,
a point $x_k$ and a surrogate $h_k$ such that
\begin{equation}
   \E[H_k(x_k)-h_k^\star] \leq C(1-\tau)^T(H_k(x_\kmone)-H_k^\star + \xi_{\kmone}) + B\sigma^2 ~~~\text{with}~~~ H_k(x) = F(x) + \frac{\kappa}{2}\|x-y_\kmone\|^2, \label{eq:newcatalyst}
\end{equation}
where $H_k$ is approximately minimized by $\mtd$, $h_k$ is a model of $H_k$ that satisfies~\HypA,~\HypB~and that can be minimized in closed form, and $\xi_{\kmone} =O(\E[F(x_\kmone)-F^\star])$;
it is easy to show that~\HypC~is also satisfied with the choice $\delta_k = C(1-\tau)^T(H_k(x_\kmone)-H_k^\star + \xi_\kmone) + B\sigma^2$ since
$\E[F(x_k)] \leq \E[H_k(x_k)] \leq \E[h_k^\star] + \delta_k$.
In other words, $\mtd$ is used to perform \emph{approximate minimization} of $H_k$, but we consider cases where
$\mtd$ also provides \emph{another surrogate} $h_k$ with closed-form minimizer that satisfies the conditions required to use Algorithm~\ref{alg:meta}, which has better convergence guarantees than Algorithm~\ref{alg:metaB} (same convergence rate up to a better factor).

As shown in Appendix~\ref{appendix:lower}, even though~(\ref{eq:newcatalyst}) looks technical, a large class of optimization
techniques are able to provide the condition~(\ref{eq:newcatalyst}), including
many variants of proximal stochastic gradient descent methods with variance reduction such as SAGA~\cite{saga}, MISO~\cite{miso}, SDCA~\cite{accsdca}, or SVRG~\cite{proxsvrg}.

Whereas~(\ref{eq:newcatalyst}) seems to be a minor modification of~(\ref{eq:aux}),
an important consequence is that it will allow us to gain a factor $1/\sqrt{q}$ in complexity when $\mu >0$, corresponding 
precisely to the sub-optimality factor.
Therefore, even though the surrogate~$H_k$ needs only be
minimized approximately, the condition~(\ref{eq:newcatalyst}) allows us to use
Algorithm~\ref{alg:meta} instead of Algorithm~\ref{alg:metaB}. The dependency with respect to $\delta_k$ being better than $\varepsilon_k$ (by $1/\sqrt{q}$), we have then the following result:
\begin{proposition}[Stochastic Catalyst with Optimality Gaps, $\mu > 0$]\label{prop:newcatalyst}
Consider Algorithm~\ref{alg:meta} with a method $\mtd$ and surrogate $h_k$ satisfying~(\ref{eq:newcatalyst}) when $\mtd$ is used to minimize $H_k$ by using $x_\kmone$ as a warm restart. 
Assume that $f$ is $\mu$-strongly convex and that there exists a parameter $\eta$ that can reduce the bias $B\sigma^2$ by a
factor~$\eta < 1$ while paying a factor $1/\eta$ in terms of inner-loop complexity.

Choose
$\delta_k=O((1\!-\!\sqrt{q}/2)^k(F(x_0)\!-\!F^\star))$ and $\eta_k=\min(1,\delta_k/(2B\sigma^2))$. Then, the complexity to solve~(\ref{eq:newcatalyst}) and compute~$x_k$ is $\tilde{O}(1/\tau\eta_k)$, and the global
complexity to obtain $\E[F(x_k)-F^\star] \leq \varepsilon$ is
\begin{equation*}
   \tilde{O}\left( \frac{1}{\tau\sqrt{q}}\log\left(\frac{F(x_0)-F^\star}{\varepsilon}\right) + \frac{{B\sigma^2}}{{q} \tau\varepsilon}\right).  
\end{equation*}
\end{proposition}
The term on the left is the accelerated rate of Catalyst for
deterministic problems, whereas the term on the right is potentially optimal
for strongly convex problems, as illustrated in the next table.
We provide indeed practical choices for the parameters $\kappa$, leading to
various values of $B,\tau,q$, for the proximal stochastic gradient descent
method with iterate averaging
as well as variants of SAGA,MISO,SVRG that can cope with stochastic perturbations,
which are discussed in Appendix~\ref{appendix:lower}. All the values below are given 
up to universal constants to simplify the presentation.

\begin{tabular}{|c|c|c|c|c|c|c|}
\hline
Method~$\mtd$ & $h_k$ &  $\kappa$ & $\tau$ & $B$ & $q$ & Complexity after Catalyst \\
\hline
prox-SGD & (\ref{eq:prox})  & $L-\mu$ & $\frac{1}{2}$ & $\frac{1}{L}$ & $\frac{\mu}{L}$ & $\tilde{O}\left(\sqrt{\frac{L}{\mu}}\log\left(\frac{F_0}{\varepsilon}\right) + \frac{\sigma^2}{\mu \varepsilon} \right)$ \\
\hline
SAGA/MISO/SVRG with $\frac{L}{n} \geq \mu$ & {(\ref{eq:newcatalyst})}  & $\frac{L}{n}-\mu$  & $\frac{1}{n}$   & $\frac{1}{L}$   & $\frac{\mu n}{L}$  &   
   $\tilde{O}\left( \sqrt{n\frac{L}{\mu}}\log\left(\frac{F_0}{\varepsilon} \right)   +  \frac{{\sigma}^2}{\mu\varepsilon} \right)$ \\
\hline
\end{tabular}

In this table, $F_0\defin F(x_0)\!-\!F^\star$ and the methods SAGA/MISO/SVRG are applied to the stochastic finite-sum
problem discussed in Section~\ref{sec:intro} with $n$ $L$-smooth functions. 
As in the deterministic case, we note that when ${L}/{n} \leq \mu$, there is no acceleration for SAGA/MISO/SVRG since the complexity of the unaccelerated method~$\mtd$ is
$\tilde{O}\left( {n}\log\left({F_0}/{\varepsilon} \right)   +  {{\sigma}^2}/{\mu\varepsilon} \right)$, which is independent of the condition number and already optimal~\cite{kulunchakov2019estimate}.
In comparison, the logarithmic terms in
$L,\mu$ that are hidden in the notation $\tilde{O}$ do not appear for a variant of the SVRG
method with direct acceleration introduced in~\cite{kulunchakov2019estimate}. Here, our approach is more generic.
Note also that $\sigma^2$ for prox-SGD and SAGA/MISO/SVRG cannot be compared to each other since the source of randomness is larger for prox-SGD, see~\cite{smiso,kulunchakov2019estimate}.


\section{Experiments}\label{sec:exp}
In this section, we perform numerical evaluations by following~\cite{kulunchakov2019estimate}, which was notably able to make SVRG and SAGA robust to stochastic noise, and accelerate SVRG. Code to reproduce the experiments is provided with the submission and more details and experiments are given in Appendix~\ref{appendix:exp}.

\vs
\paragraph{Formulations.}
Given training data
$(a_i,b_i)_{i=1,\ldots,n}$, with $a_i$ in $\Real^p$ and $b_i$ in $\{-1,+1\}$, we consider the optimization problem
\begin{displaymath}
    \min_{x \in \Real^p} \frac{1}{n}\sum_{i=1}^n \phi(b_i a_i^\top x)  + \frac{\mu}{2}\|x\|^2,
\end{displaymath}
where $\phi$ is either the logistic loss $\phi(u)=\log(1+e^{-u})$, or the squared hinge loss $\phi(u)=\frac{1}{2}\max(0,1-u)^2$, which are both $L$-smooth, with $L=0.25$ for logistic and $L=1$ for the squared hinge loss.
Studying the squared hinge loss is interesting since its gradients are unbounded on the optimization domain, which may break the bounded noise assumption.
The regularization parameter $\mu$ acts as the strong convexity constant for the problem and is chosen among the smallest values one would try when performing parameter search, \eg, by cross validation. Specifically, we consider $\mu=1/10n$ and $\mu=1/100n$, where $n$ is the number of training points; we also try $\mu=1/1000n$ to evaluate the numerical stability of methods in very ill-conditioned problems.
Following~\cite{smiso,kulunchakov2019estimate,zheng2018lightweight}, we
consider DropOut perturbations~\citep{wager2014}---that is, setting each component~$(\nabla f(x))_i$ to 0 with a probability~$\delta$ and to~$(\nabla f(x))_i/(1-\delta)$ otherwise. This procedure is motivated by the need of a simple optimization benchmark illustrating stochastic finite-sum problems, where the amount of perturbation is easy to control. The settings used in our experiments are~$\delta=0$ (no noise) and~$\delta\in\{0.01,0.1\}$.

\paragraph{Datasets.}
We consider three datasets with various number of points~$n$ and dimension~$p$. All the data points are normalized to have unit $\ell_2$-norm.
The description comes from~\cite{kulunchakov2019estimate}:
\begin{itemize}[leftmargin=*]
\item \textrm{alpha} is from  the  Pascal  Large  Scale Learning Challenge
website\footnote{\url{http://largescale.ml.tu-berlin.de/}} and contains $n=250\,000$
points in dimension $p=500$.
\item \textrm{gene} consists of gene expression data and the binary labels $b_i$ characterize two different types of breast cancer. This is a small dataset with $n=295$ and $p=8\,141$.
\item \textrm{ckn-cifar} is an image classification task where each image from the CIFAR-10 dataset\footnote{\url{https://www.cs.toronto.edu/~kriz/cifar.html}} is represented by using a two-layer unsupervised convolutional neural network~\citep{mairal2016end}. We consider here the binary classification task consisting of predicting the class 1 vs. other classes, and use our algorithms for the classification layer of the network, which is convex. The dataset contains $n=50\,000$ images and the dimension of the representation is $p=9\,216$. 
\end{itemize}

\paragraph{Methods.}
We consider the variants of SVRG and SAGA of~\cite{kulunchakov2019estimate},
which use decreasing step sizes when $\delta > 0$ (otherwise, they do not
converge). We use the suffix ``-d'' each time decreasing step sizes are used.
We also consider Katyuasha~\cite{accsvrg} when $\delta=0$, and the accelerated
SVRG method of~\cite{kulunchakov2019estimate}, denoted by acc-SVRG.
Then, SVRG-d, SAGA-d, acc-SVRG-d are used with the step size strategies
described in~\cite{kulunchakov2019estimate}, by using the code provided to us
by the authors.

\vs
\paragraph{Practical questions and implementation.}
In all setups, we choose the parameter~$\kappa$ according to theory, which are described in the previous section, following Catalyst~\cite{catalyst_jmlr}.
For composite problems, Proposition~\ref{prop:restart} suggests to use $x_\kmone$ as a warm start for inner-loop problems. For smooth ones,~\cite{catalyst_jmlr} shows that in fact, other choices such as $y_\kmone$ are appropriate and lead to similar complexity results. In our experiments with smooth losses, we use $y_\kmone$, which has shown to perform consistently better. 

The strategy for $\eta_k$ discussed in Proposition~\ref{prop:newcatalyst} suggests to use constant step-sizes for a while in the inner-loop, typically of order $1/(\kappa+L)$ for the methods we consider, before using an exponentially decreasing schedule. Unfortunately, even though theory suggests a rate of decay in $(1-\sqrt{q}/2)^k$, it does not provide useful insight on when decaying should start since the theoretical time requires knowing $\sigma^2$. A similar issue arise in stochastic optimization techniques involving iterate averaging \cite{bottou2018optimization}. We adopt a similar heuristic as in this literature and start decaying after $k_0$ epochs, with $k_0=30$.
Finally, we discuss the number of iterations of $\mtd$ to perform in the inner-loop. When $\eta_k=1$, the theoretical value is of order $\tilde{O}(1/\tau)=\tilde{O}(n)$, and we choose exactly $n$ iterations (one epoch), as in Catalyst~\cite{catalyst_jmlr}. After starting decaying the step-sizes ($\eta_k < 1$), we use $\lceil n/\eta_k\rceil$, according to theory.

\vs
\paragraph{Experiments and conclusions.}
We run each experiment five time with a different random seed and average the results. All curves also display one standard deviation.
Appendix~\ref{appendix:exp} contains numerous experiments, where we vary the amount of noise, the type of approach (SVRG vs. SAGA), the amount of regularization $\mu$, and choice of loss function. In Figure~\ref{fig:exp}, we show a subset of these curves. Most of them show that acceleration may be useful even in the stochastic optimization regime, consistently with~\cite{kulunchakov2019estimate}. {At the same time, all acceleration methods may not perform well for very ill-conditioned problems with $\mu=1/1000n$, where the sublinear convergence rates for convex optimization ($\mu = 0$) are typically better than the linear rates for strongly convex optimization ($\mu > 0$). However, these ill-conditioned cases are often unrealistic in the context of empirical risk minimization.
\begin{figure}[!hptb]
\vspace*{-0cm}
  \includegraphics[width=\linewidth]{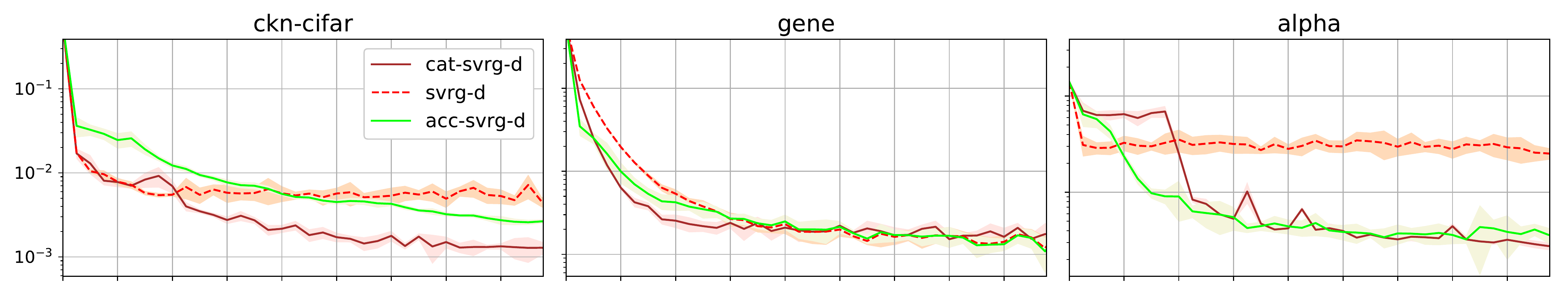}\vspace{-0.19cm}\\
  \includegraphics[width=\linewidth]{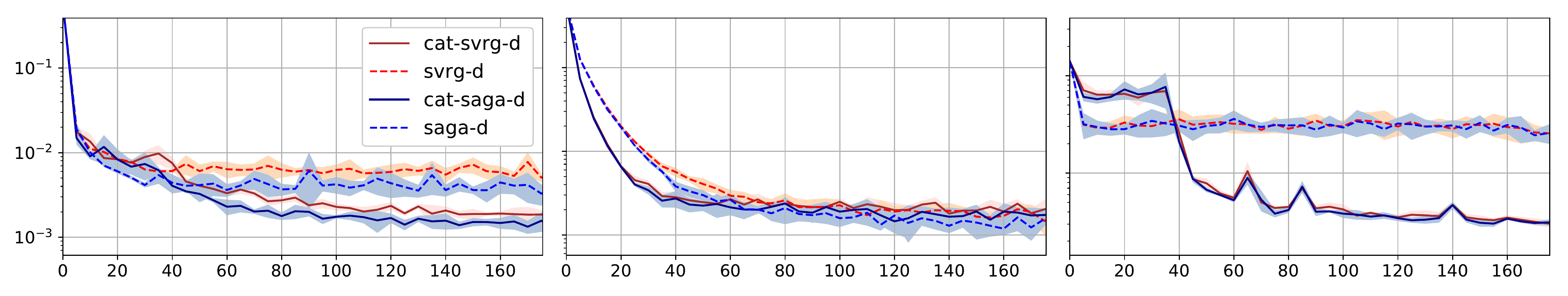}\\
  \vspace*{-0.6cm}
  \caption{Accelerating SVRG-like (top) and SAGA (bottom)  methods for $\ell_2$-logistic regression with $\mu = 1/(100 n)$  (bottom) for $\delta=0.1$.
  All plots are on a logarithmic scale for the objective function value, and the $x$-axis denotes the number of epochs.
  The colored tubes around each curve denote a standard deviations across $5$ runs. They do not look symmetric because of the logarithmic scale.
  }\label{fig:exp}
\end{figure}


\newpage
\subsection*{Acknowledgments}

This work was supported by the ERC grant SOLARIS (number 714381) and ANR 3IA MIAI@Grenoble Alpes. The authors would like to thank Anatoli Juditsky for numerous interesting discussions that greatly improved the quality of this manuscript.

{
\small
\bibliographystyle{abbrv}
\bibliography{main}
}

\newpage

\appendix

\section{Useful Results and Definitions}\label{appendix:useful}
In this section, we present auxiliary results and definitions.
\begin{definition}[Proximal operator]
Given a convex lower-semicontinuous function $\psi$ defined on $\Real^p$, the proximal operator of $\psi$ is defined as the unique solution of the strongly-convex problem
\begin{displaymath}
    \text{Prox}_{\psi}[y] = \argmin_{x \in \Real^p} \left\{ \frac{1}{2}\|y-x\|^2 + \psi(x) \right\}.
\end{displaymath}
\end{definition}

\begin{lemma}[Convergence rate of the sequences $(\alpha_k)_{k \geq 0}$ and~$(A_k)_{k \geq 0}$]\label{lemma:Ak}
Consider the sequence in $(0,1)$ defined by the recursion 
$$ \alpha_k^2 = (1-\alpha_k)\alpha_{\kmone}^2 + q \alpha_k~~~\text{with}~~~~0 \leq q < 1,$$
and define $A_k = \prod_{t=1}^k (1-\alpha_t)$. Then, 
\begin{itemize}
   \item if $q=0$ and $\alpha_0=1$, then, for all $k \geq 1$,
   $$\frac{2}{(k+2)^2} \leq A_k = \alpha_k^2 \leq  \frac{4}{(k+2)^2}.$$
   \item if $\alpha_0 = \sqrt{q}$, then  for all $k \geq 1$,
   $$A_k = (1-\sqrt{q})^k ~~~~\text{and}~~~~ \alpha_k = \sqrt{q}.$$
   \item if $\alpha_0 = 1$, then  for all $k \geq 1$, 
   $$A_k \leq \min \left( (1-\sqrt{q})^k, \frac{4}{(k+2)^2}\right) ~~~~\text{and}~~~~ \alpha_k \geq \max\left(\sqrt{q},\frac{\sqrt{2}}{k+2}\right).$$
\end{itemize}
\end{lemma}
\begin{proof}
We prove the three points, one by one.
   
\paragraph{First point.}
Let us prove the first point when $q=0$ and $\alpha_0=1$. The relation $A_k = \alpha_k^2$ is obvious for all $k \geq 1$ and the relation 
$\alpha_k^2  \leq  \frac{4}{(k+2)^2}$ holds for $k=0$. By induction, let us assume that we have the relation 
$\alpha_{\kmone}^2  \leq  \frac{4}{(k+1)^2}$ and let us show that it propagates for $\alpha_k^2$.
Assume, by contradiction, that $\alpha_k^2 > \frac{4}{(k+2)^2}$, meaning that $\alpha_k > \frac{2}{(k+2)}$.
Then,
\begin{displaymath}
   \alpha_k^2 = (1-\alpha_k)\alpha_{\kmone}^2 \leq (1-\alpha_k)\frac{4}{(k+1)^2} < \frac{4 k}{ (k+2)(k+1)^2} =\frac{4}{ (k+2)(k+2 + \frac{1}{k})} < \frac{4}{(k+2)^2},
\end{displaymath}
and we obtain a contradiction. Therefore, $\alpha_k^2 \leq \frac{4}{(k+2)^2}$ and the induction hypothesis allows us to conclude for all $k \geq 1$.
Then, note~\cite{paquette2017catalyst} that we also have for all $k \geq 1$,
\begin{displaymath}
   A_k = \prod_{t=1}^k (1-\alpha_t) \geq \prod_{t=1}^k \left(1-\frac{2}{t+2}\right) = \frac{2}{(k+1)(k+2)} \geq \frac{2}{(k+2)^2}.
\end{displaymath}

\paragraph{Second point.} The second point is obvious by induction.
 
 \paragraph{Third point.}
For the third point, we simply assume $\alpha_0=1$ such that $\alpha_0 \geq \sqrt{q}$. Then, the relation $\alpha_k \geq \sqrt{q}$ and therefore $A_k \leq (1-\sqrt{q})^k$ are easy to show by induction.
   Then, consider the sequence defined recursively by $u_k^2 = (1-u_k)u_{\kmone}^2$ with $u_0=1$. From the first point, we have that $\frac{\sqrt{2}}{k+2} \leq u_k \leq\frac{{2}}{k+2}$.
   We will show that $\alpha_k \geq u_k$ for all $k \geq 0$, which will be sufficient to conclude since then we would have $A_k \leq \prod_{t=1}^k (1-u_t) \leq \frac{4}{(k+2)^2}$.
   First, we note that $\alpha_0 = u_0$; then, assume that $\alpha_{\kmone} \geq u_{\kmone}$ and also assume by contradiction that $\alpha_k > u_k$.
   This implies that
   $$ u_k^2 = (1-u_k) u_{\kmone}^2  \leq (1-u_k) \alpha_{\kmone}^2 < (1-\alpha_k) \alpha_{\kmone}^2 \leq  \alpha_k^2,$$
   which contradicts the assumption $\alpha_k > u_k$. This allows us to conclude by induction.
\end{proof}

\begin{lemma}[Convergence rate of sequences $\Theta_k=\prod_{i=1}^k(1-\theta_i)$]\label{lemma:theta}
Consider the sequence $\theta_j=\frac{\gamma}{(1+j)^{1+\gamma}}$ with $\gamma$ in $(0,1]$. Then,
\begin{equation}
   e^{-(1+\gamma)} \leq \Theta_k \leq 1. \label{eq:theta}
\end{equation}
\end{lemma}
\begin{proof}
We use the classical inequality $\log(1+u) \geq \frac{u}{1+u}$ for all $u > -1$:
\begin{displaymath}
-\log(\Theta_k) = - \sum_{j=1}^k \log\left( 1 -\frac{\gamma}{(1+j)^{1+\gamma}}\right) 
               \leq \sum_{j=1}^k \frac{\gamma}{(1+j)^{1+\gamma} - \gamma}  \leq \sum_{j=1}^k \frac{\gamma}{j^{1+\gamma}}, 
\end{displaymath}
when noting that the function $g(x) = (1+x)^{1+\gamma} - x^{1+\gamma}$ is greater than $\gamma$ for all $x \geq 1$, since $g(1) \geq 1 \geq \gamma$ and $g$ is non-decreasing.
Then,
\begin{displaymath}
 -\log(\Theta_k) \leq \sum_{j=1}^k \frac{\gamma}{j^{1+\gamma}} \leq \gamma +  \gamma \int_{x=1}^k \frac{1}{x^{1+\gamma}} dx = \gamma + 1 - \frac{1}{k^\gamma} \leq \gamma+1.
\end{displaymath}
Then, we immediately obtain~(\ref{eq:theta}).

\end{proof}
\section{Details about Complexity Results}\label{appendix:restart}

\subsection{Details about~(\ref{eq:restart2})}\label{appendix:sec:restart1}
Consider the complexity~(\ref{eq:aux}) with $h=F$. To achieve the accuracy $2B\sigma^2$, it is sufficient to run the method~$\mtd$ for~$t_0$ iterations, such that 
$$  C (1-\tau)^{t_0}(F(x_0)-F^\star)  \leq B\sigma^2.$$  
It is then easy to see that this inequality is satisfied as soon as $t_0$ is greater than $\frac{1}{\tau}\log(C (F(x_0)-F^\star)/B\sigma^2)$. 
Since $\varepsilon \leq B\sigma^2$ and using the concavity of the logarithm function, it is also sufficient to choose $t_0= \frac{1}{\tau}\log(C (F(x_0)-F^\star)/\varepsilon)$.

Then, we perform $K$ restart stages such that $\varepsilon_K \leq \varepsilon$. 
Each stage is initialized with a point $x_k$ satisfying $\E[F(x_k)-F^\star] \leq \varepsilon_\kmone$, and the goal of each stage is to reduce the error by a factor $1/2$.
Given that $\eta_k$ increases the computational cost, the complexity of the $k$-th stage is then upper-bounded by
$\frac{2^k }{\tau} \log(2C)$, leading to the global complexity
\begin{equation*} 
O\left( \frac{1}{\tau} \log\left( \frac{C(F(x_0)-F^\star)}{\varepsilon}  \right)+ \sum_{k=1}^K \frac{2^k}{\tau} \log\left( { 2C}  \right)         \right) ~~~\text{with}~~~ K = \left \lceil \log_2\left(   \frac{2B\sigma^2}{\varepsilon}  \right)\right\rceil, 
\end{equation*} 
and~(\ref{eq:restart2}) follows by elementary calculations.

\subsection{Obtaining~(\ref{eq:restart2}) from~(\ref{eq:aux2})}
Since $h$ is $\mu$-strongly convex, we notice that~(\ref{eq:aux2}) implies the rate
 \begin{equation*}
   \E[h(z_t) - h^\star] \leq \frac{D (h(z_0)-h^\star)}{\mu t^d} + \frac{B \sigma^2}{2}, 
\end{equation*}
by using the strong convexity inequality $h(z_0)\geq h^\star +
\frac{\mu}{2}\|z_0-z^\star\|^2$.
After running the algorithm for $t'= \lceil (2D/\mu)^{1/d} \rceil$ iterations, we can show that
 \begin{equation*}
   \E[h(z_{t'}) - h^\star] \leq \frac{h(z_0)-h^\star}{2} + \frac{B \sigma^2}{2}.
\end{equation*}
Then, when restarting the procedure $s$ times (using the solution of the
previous iteration as initialization), and denoting by $h_{st'}$ the last iterate, it is easy to show that
 \begin{equation*}
   \E[h(x_{st'}) - h^\star] \leq \frac{h(x_0)-h^\star}{2^s} + \frac{B \sigma^2}{2}\left( \sum_{i=0}^{s-1} \frac{1}{2^i}\right) \leq \frac{h(z_0)-h^\star}{2^s} + B \sigma^2.
\end{equation*}
Then, calling $t=st'$, we can use the inequality $2^{-u} \leq 1 - \frac{u}{2}$ for $u$ in $[0,1]$, due to convexity, and
 \begin{equation*}
   \E[h(z_{t}) - h^\star] \leq (h(z_0)-h^\star)\left({2^{-1/t'}}\right)^t + B \sigma^2 =  (h(z_0)-h^\star)\left(1 - \frac{1}{2t'} \right)^k + B \sigma^2,
\end{equation*}
which gives us~(\ref{eq:aux}) with $C=1$ and $\tau=\frac{1}{2t'}$.
It is then easy to obtain~(\ref{eq:restart2}) by following similar steps as in
Section~\ref{appendix:sec:restart1}, by noticing that the restart frequency is
of the same order $O(1/\tau)$.

\subsection{Details about~(\ref{eq:stoch_catalyst})}\label{appendix:sec:stoch_catalyst}
\paragraph{Inner-loop complexity.}
Since $\eta_k$ is chosen such that the bias $\eta_kB\sigma^2$ is smaller than $\varepsilon_k$, the number of iterations of $\mtd$ to solve
the sub-problem is $\tilde{O}(\tau) = O(\log(1/q)\tau)$, as in the deterministic case, and the complexity is thus $\tilde{O}(\tau/\eta_k)$.

\paragraph{Outer-loop complexity.}
Since $\E[F(x_k)-F^\star] \leq O((1-\sqrt{q}/3)^k(F(x_0)-F^\star))/q$ according to Proposition~\ref{prop:B}, it suffices to choose 
\begin{displaymath}
    K = O\left(\frac{1}{\sqrt{q}}\log\left(\frac{F(x_0)-F^\star}{q\varepsilon}\right)\right)
\end{displaymath}
iterations to guarantee $\E[F(x_K)-F^\star] \leq \varepsilon = O( \varepsilon_K/q) = O((1-\sqrt{q}/3)^K(F(x_0)-F^\star)/q)$.

\paragraph{Global complexity.}
The total complexity to guarantee $\E[F(x_k)-F^\star] \leq \varepsilon$ is then
\begin{displaymath}
\begin{split}
    C & = \sum_{k=1}^K \tilde{O}\left(\frac{\tau}{\eta_k}\right) \\ 
    & \leq   \tilde{O}\left( \sum_{k=1}^K {\tau} + \sum_{k=1}^K \frac{{B\sigma^2 \tau}}{\varepsilon_k}\right) \\
    & =   \tilde{O}\left( \sum_{k=1}^K {\tau} + \sum_{k=1}^K \frac{{B\sigma^2 \tau}}{\left(1-\frac{\sqrt{q}}{3}\right)^k(F(x_0)-F^\star)}\right) \\
    & =   \tilde{O}\left( \frac{\tau}{\sqrt{q}}\log\left(\frac{F(x_0)-F^\star}{\varepsilon}\right) + \frac{{B\sigma^2 \tau}}{\sqrt{q}\left(1-\frac{\sqrt{q}}{3}\right)^{K+1}(F(x_0)-F^\star)}\right) \\
    & =   \tilde{O}\left( \frac{\tau}{\sqrt{q}}\log\left(\frac{F(x_0)-F^\star}{\varepsilon}\right) + \frac{{B\sigma^2 \tau}}{{q^{3/2}}\varepsilon}\right), \\
\end{split}
\end{displaymath}
where the last relation uses the fact that $\varepsilon = O( \varepsilon_K/q) = O((1-\sqrt{q}/3)^K(F(x_0)-F^\star)/q)$.

\subsection{Complexity of accelerated stochastic proximal gradient descent with inexact prox}
Assume that $h_k(x_k)-h_k^\star \leq \varepsilon_k$. Then, following similar steps as in~(\ref{eq:sgd2}),
\begin{displaymath}
\begin{split}
  \E[F(x_k)] & \leq \E[h_k(x_k)] + \E[(g_k-\nabla f(y_\kmone))^\top (x_k - y_\kmone)] \\ 
   & = \E[h_k(x_k)] + \E[(g_k-\nabla f(y_\kmone))^\top (x_k - u_\kmone)] \\ 
   & = \E[h_k(x_k)] + \E[(g_k-\nabla f(y_\kmone))^\top (x_k - x_k^\star)]  + \E[(g_k-\nabla f(y_\kmone))^\top (x_k^\star-u_\kmone)] \\ 
   & \leq \E[h_k(x_k)] + \E[(g_k-\nabla f(y_\kmone))^\top (x_k - x_k^\star)]  + \frac{\sigma^2}{\kappa+\mu} \\ 
   & \leq \E[h_k(x_k)] + \frac{\E[\|g_k-\nabla f(y_\kmone)\|^2]}{2(\kappa+\mu)} + \frac{(\kappa+\mu)\E[\|x_k - x_k^\star\|^2]}{2}  + \frac{\sigma^2}{\kappa+\mu} \\ 
   & \leq \E[h_k(x_k)] + \E[h_k(x_k) - h_k^\star]  + \frac{3\sigma^2}{2(\kappa+\mu)} \\ 
   & \leq \E[h_k^\star] + 2 \varepsilon_k + \frac{3\sigma^2}{2(\kappa+\mu)}. 
\end{split}
\end{displaymath}
And thus, $\delta_k = 2 \varepsilon_k + \frac{3\sigma^2}{2(\kappa+\mu)}$.

\section{Proofs of Main Results}\label{appendix:proofs}
\subsection{Proof of Propositions~\ref{prop:A} and~\ref{prop:B}}
\begin{proof}
In order to treat both propositions jointly, we introduce the quantity
\begin{displaymath}
   w_k = \left\{ \begin{array}{ll}
      x_k^\star & \text{for Algorithm}~1 \\
      x_k & \text{for Algorithm}~2 \\
      \end{array}\right.,
\end{displaymath}
and, for all $k \geq 1$,
\begin{equation}
   v_k = w_k + \frac{1-\alpha_{\kmone}}{\alpha_\kmone}(w_k - x_\kmone), \label{eq:defs}
\end{equation}
with $v_0=x_0$, as well as $\eta_k = \frac{\alpha_k-q}{1-q}$ for all $k \geq 0$.

Note that the following relations hold for all $k \geq 1$, keeping in mind that $\alpha_k^2=(1-\alpha_k)\alpha_\kmone^2 + q\alpha_k$:
\begin{equation*}
\begin{split}
1-\eta_k & = \frac{1-\alpha_k}{1-q} = \frac{(\kappa+\mu)(1-\alpha_k)}{\kappa} \\
\eta_k & = \frac{\alpha_k - q}{1-q} = \frac{\alpha_k^2 - q\alpha_k}{\alpha_k -q\alpha_k}  = \frac{\alpha_\kmone^2(1-\alpha_k)}{\alpha_k -  \alpha_k^2 + (1-\alpha_k)\alpha_\kmone^2} = \frac{\alpha_\kmone^2}{\alpha_\kmone^2 + \alpha_k}. 
\end{split}
\end{equation*}
Then, based on the previous relations, we have
\begin{equation*}
\begin{split}
 y_k & = w_k + \beta_k (w_k - x_\kmone) + \frac{(\kappa+\mu)(1-\alpha_k)}{\kappa}(x_k -w_k) \\
  & = w_k + \frac{\alpha_\kmone(1-\alpha_\kmone)}{\alpha_\kmone^2 + \alpha_k} (w_k - x_\kmone) + (1-\eta_k)(x_k -w_k) \\
  & = w_k + \frac{\eta_k(1-\alpha_\kmone)}{\alpha_\kmone} (w_k - x_\kmone) + (1-\eta_k)(x_k -w_k) \\
  & = \eta_k v_k + (1-\eta_k)x_k,
\end{split}
\end{equation*}
which is similar to the relation used in~\cite{catalyst_jmlr} when $w_k=x_k$.
Then, the proof differs from~\cite{catalyst_jmlr} since we introduce the surrogate function $h_k$. For all $x$ in $\Real^p$,
\begin{equation}
\begin{split}
    h_k(x) & \geq h_k^\star + \frac{\kappa+\mu}{2}\|x -x_k^\star\|^2 ~~~\text{(by strong convexity, see \hypA)}\\
           & = h_k^\star + \frac{\kappa+\mu}{2}\|x-w_k\|^2 + \underbrace{\frac{\kappa+\mu}{2}\|w_k-x_k^\star\|^2 + (\kappa+\mu)\langle x-w_k, w_k-x_k^\star\rangle}_{- \Delta_k(x)}. 
\end{split} \label{eq:auxhk}
\end{equation}
Introduce now the following quantity for the convergence analysis: 
$$z_{\kmone} = \alpha_\kmone x^\star + (1-\alpha_\kmone) x_\kmone,$$
and consider $x=z_\kmone$ in~(\ref{eq:auxhk}) while taking expectations, noting that  all random variables indexed by $\kmone$ are deterministic given $\Fcal_\kmone$,
\begin{equation}
\begin{split}
     \E[F(x_k)] & \leq \E[h_k^\star] + \delta_k \qquad \text{(by \hypC)}\\
                & \leq \E[ h_k(z_\kmone)] -\E\left[ \frac{\kappa+\mu}{2}\|z_\kmone-w_k\|^2 \right] + \E[\Delta_k(z_\kmone)] + \delta_k \\
                & \leq \E[F(z_\kmone)] + \E\left[ \frac{\kappa}{2}\|z_\kmone-y_\kmone\|^2 \right] -\E\left[ \frac{\kappa+\mu}{2}\|z_\kmone-w_k\|^2 \right] + \E[\Delta_k(z_\kmone)] + \delta_k,
\end{split}\label{eq:sufficient_descent}
\end{equation}
where the last inequality is due to \HypB.

Let us now open a parenthesis and derive a few relations that will be useful to find a Lyapunov function.
To use more compact notation, define $X_k = \E[\|x^\star- x_k\|^2]$, $V_k = \E[\|x^\star- v_k\|^2]$ and $F_k = \E[F(x_k)-F^\star]$, and note that
\begin{equation}
\begin{split}
  \E[F(z_\kmone)] & \leq \alpha_\kmone f^\star + (1-\alpha_\kmone) \E[F(x_\kmone)] - \frac{\mu \alpha_\kmone(1-\alpha_\kmone)}{2}X_\kmone \\
  \E[\|z_\kmone - w_k\|^2] & = \alpha_\kmone^2 V_k \\
  \E[\|z_\kmone - y_\kmone\|^2] & \leq \alpha_\kmone(\alpha_\kmone - \eta_\kmone) X_{\kmone} + \alpha_\kmone \eta_\kmone V_\kmone. \\
\end{split}\label{eq:relations}
\end{equation}
The first relation is due to the convexity of $f$; the second one can be
obtained from the definition of~$v_k$ in~(\ref{eq:defs}) after simple
calculations; the last one can be obtained as in the proof of Theorem 3
in~\cite{catalyst_jmlr} (end of page 16). 

We may now come back to~(\ref{eq:sufficient_descent}) and we use the relations~(\ref{eq:relations}):
\begin{multline*}
   F_k + \frac{(\kappa+\mu)\alpha_{\kmone}^2}{2}V_k   \leq  (1-\alpha_\kmone)F_{k-1} - \frac{\mu \alpha_\kmone(1-\alpha_\kmone)}{2}X_\kmone+ \\ \frac{\kappa}{2}
   \alpha_\kmone(\alpha_\kmone - \eta_\kmone) X_{\kmone} + \frac{\kappa}{2}\alpha_\kmone \eta_\kmone V_\kmone 
     +  {{\delta}_k} + \E[\Delta_k(z_\kmone)]. 
\end{multline*}
It is then easy to see that the terms involving $X_\kmone$ cancel each other since $\eta_\kmone=\alpha_\kmone-\frac{\mu}{\kappa}(1-\alpha_\kmone)$.

\paragraph{Lyapunov function.}
We may finally define the Lyapunov function 
\begin{equation}
   S_k = (1-\alpha_k)F_k + \frac{\kappa \alpha_k \eta_k}{2}V_k. \label{eq:Sk}
\end{equation}
and we obtain
 \begin{equation}
   \frac{S_k}{1-\alpha_k} \leq S_{k-1} 
       + {{\delta}_k} + \E[\Delta_k(z_\kmone)], \label{eq:lyapunov}
\end{equation}
For variant Algorithm~\ref{alg:meta}, we have $\Delta_k(z_\kmone)=0$ since
$w_k=x_k^\star$, and we obtain the following relation by unrolling the
recursion:
\begin{equation}
   S_k \leq {A_k}\left (S_0 + \sum_{j=1}^k \frac{\delta_j}{A_{j-1}} \right) \qquad \text{with} \qquad A_j = \prod_{i=1}^j(1-\alpha_i). \label{eq:propA}
\end{equation}

\paragraph{Specialization to $\mu > 0$.}
When $\mu > 0$, we have $\alpha_0=\sqrt{q}$ and 
\begin{equation}
\begin{split}
    S_0 & = (1-\sqrt{q})(F(x_0)-F^\star) + \frac{\kappa \sqrt{q} (\sqrt{q}-q)}{2(1-q)}\|x_0-x^\star\|^2 \\
        & = (1-\sqrt{q})(F(x_0)-F^\star) + \frac{(\kappa+\mu) \sqrt{q} (\sqrt{q}-q)}{2}\|x_0-x^\star\|^2 \\
        & = (1-\sqrt{q})(F(x_0)-F^\star) + \frac{\mu(1-\sqrt{q})}{2}\|x_0-x^\star\|^2 \\
        & \leq 2(1-\sqrt{q})(F(x_0)-F^\star),
\end{split}\label{eq:S0}
\end{equation}
by using the strong convexity inequality $F(x_0) \geq F^\star + \frac{\mu}{2}\|x_0-x^\star\|^2$.
Then, noting that $\E[F(x_k)-F^\star] \leq \frac{S_k}{1-\sqrt{q}}$ and $A_k=(1-\sqrt{q})^k$ (Lemma~\ref{lemma:Ak}), we immediately obtain the first part of~(\ref{eq:propAb}) from~(\ref{eq:propA}).

\paragraph{Specialization to $\mu = 0$.}
When $\mu =0$, we have $\alpha_0=1$ and  $S_0 = \frac{\kappa}{2} \|x_0-x^\star\|^2$.
Then, according to Lemma~\ref{lemma:Ak} and~(\ref{eq:propA}), for $k \geq 1$,
\begin{equation}
   \E[F(x_k)-F^\star] \leq \frac{S_k}{1-\alpha_k} \leq \frac{\kappa \|x_0-x^\star\|^2}{2}A_\kmone + \sum_{j=1}^k \frac{\delta_jA_\kmone}{A_{j-1}},
\end{equation}
and we obtain the second part of~(\ref{eq:propAb}) noting that $A_\kmone \leq \frac{4}{(k+1)^2}$ and that $A_\jmone \geq \frac{2}{(j+1)^2}$.
Then, Proposition~\ref{prop:A} is proven.

\paragraph{Proof of Proposition~\ref{prop:B}.}
When $w_k=x_k$, we need to control the quantity $\Delta_k(z_\kmone)$. Consider any scalar $\theta_k$ in $(0,1)$. Then,
\begin{displaymath}
\begin{split}
   \Delta_k(z_\kmone) & = - \frac{\kappa+\mu}{2}\|x_k-x_k^\star\|^2  - (\kappa+\mu)\langle z_\kmone - x_k, x_k - x_k^\star \rangle \\
                 & = - \frac{\kappa+\mu}{2}\|x_k-x_k^\star\|^2  -  (\kappa+\mu)\alpha_\kmone \langle x^\star - v_k, x_k - x_k^\star \rangle \\
                 & \leq - \frac{\kappa+\mu}{2}\|x_k-x_k^\star\|^2  +  (\kappa+\mu)\alpha_\kmone \|x^\star-v_k\| \| x_k - x_k^\star\| \\
       & \leq   \left(\frac{1}{\theta_k}-1\right)\frac{\kappa+\mu}{2}\|x_k-x_k^\star\|^2  + \frac{\theta_k(\kappa+\mu)\alpha_\kmone^2}{2}\|x^\star-v_k\|^2   ~~\text{(Young's inequality)} \\
       & \leq   \left(\frac{1}{\theta_k}-1\right)(h_k(x_k)-h_k^\star)  + \frac{\theta_k(\kappa+\mu)\alpha_\kmone^2}{2}\|x^\star-v_k\|^2   ~~~~\text{(since $\theta_k \leq 1$)} \\
       & \leq   \left(\frac{1}{\theta_k}-1\right)(h_k(x_k)-h_k^\star)  + \frac{\theta_k(\kappa+\mu)(\alpha_k^2 - \alpha_k q)}{2(1-\alpha_k)}\|x^\star-v_k\|^2    \\
       & =   \left(\frac{1}{\theta_k}-1\right)(h_k(x_k)-h_k^\star)  + \frac{\theta_k\kappa\alpha_k\eta_k}{2(1-\alpha_k)}\|x^\star-v_k\|^2.  
   \end{split}
\end{displaymath}
Then, we take expectations and, noticing that the quadratic term involving $\|x^\star-v_k\|^2$ is smaller than $\theta_k S_k/(1-\alpha_k)$ in expectation (from the definition of $S_k$ in~(\ref{eq:Sk})), we obtain
\begin{displaymath}
 \E[\Delta_k(z_\kmone)] \leq \left(\frac{1}{\theta_k}-1\right)\varepsilon_k + \frac{\theta_k S_k}{1-\alpha_k},
\end{displaymath}
and from~(\ref{eq:lyapunov}),
$$  S_k \leq \frac{(1-\alpha_k)}{(1-\theta_k)}\left(S_{\kmone}   + \delta_k + \left(\frac{1}{\theta_k}-1\right)\varepsilon_k\right).$$
By unrolling the recursion, we obtain 
\begin{equation} 
S_k \leq \frac{A_k}{\Theta_k}\left
  (S_0 + \sum_{j=1}^k \frac{\Theta_{j-1}}{A_{j-1}}\left( \delta_j - \varepsilon_j + \frac{\varepsilon_j}{\theta_j}  \right) 
  \right)~~~\text{with}~~~ A_j = \prod_{i=1}^j(1-\alpha_i)~~~\text{and}~~~
  \Theta_j = \prod_{i=1}^j(1-\theta_i). \label{eq:propB} 
\end{equation}
\paragraph{Specialization to $\mu > 0$.}
When $\mu > 0$, we have $\alpha_k=\sqrt{q}$ for all $k \geq 0$. Then, we may choose $\theta_k=\frac{\sqrt{q}}{2}$; then, $1-\sqrt{q} \leq \left(1-\frac{\sqrt{q}}{2}\right)^2$ and
 $\frac{A_k}{\Theta_k} \leq
\big(1-\frac{\sqrt{q}}{2}\big)^k$ for all $k \geq 0$. By using the relation~(\ref{eq:S0}), we obtain
\begin{displaymath}
\begin{split}
   S_k & \leq 2\left(1-\frac{\sqrt{q}}{2}\right)^k (1-\sqrt{q})(F(x_0)-F^\star) + 2\sum_{j=1}^k \left(\frac{1-\sqrt{q}  }{1-\frac{\sqrt{q}}{2}}\right)^{k-j+1}\left( \delta_j - \varepsilon_j +  \frac{\varepsilon_j}{\sqrt{q}}\right) \\
       & \leq (1-\sqrt{q}) \left(2\left(1-\frac{\sqrt{q}}{2}\right)^k(F(x_0)-F^\star) + 4\sum_{j=1}^k \left(\frac{1-\sqrt{q}  }{1-\frac{\sqrt{q}}{2}}\right)^{k-j}\left( \delta_j - \varepsilon_j +  \frac{\varepsilon_j}{\sqrt{q}}\right)  \right)\\
       & \leq (1-\sqrt{q}) \left(2\left(1-\frac{\sqrt{q}}{2}\right)^k(F(x_0)-F^\star) + 4\sum_{j=1}^k \left({1-\frac{\sqrt{q}}{2}}\right)^{k-j}\left( \delta_j - \varepsilon_j +  \frac{\varepsilon_j}{\sqrt{q}}\right)   \right),
\end{split}
\end{displaymath}
where the second inequality uses $\frac{1}{1-\frac{\sqrt{q}}{2}} \leq 2$.
Since $(1-\sqrt{q})\E[F(x_k)-F^\star] \leq S_k$, we obtain the first part of Proposition~(\ref{prop:B}).
\paragraph{Specialization to $\mu = 0$.}
When $\mu =0$, we have $\alpha_0=1$ and  $S_0 = \frac{\kappa}{2} \|x_0-x^\star\|^2$. We may then choose $\theta_k = \frac{\gamma}{(k+1)^{1+\gamma}}$ for any $\gamma$ in $(0,1]$, 
leading to $e^{-(1+\gamma)} \leq \Theta_k \leq 1$ for all $k \geq 0$ according to Lemma~\ref{lemma:theta}.
Besides, according to the proof of Lemma~\ref{lemma:Ak}, $\frac{2}{(k+2)^2} \leq A_k \leq \frac{4}{(k+2)^2}$ for all $k \geq 1$.

Then, from~(\ref{eq:propB}),
\begin{displaymath}
\begin{split}
 \E[F(x_k)-F^\star] & \leq \frac{A_\kmone}{\Theta_k}\frac{\kappa\|x_0-x^\star\|^2}{2}  +  \sum_{j=1}^k \frac{A_\kmone \Theta_\jmone}{\Theta_k A_\jmone}\left(\delta_j - \varepsilon_j + \frac{\varepsilon_j}{\gamma}(1+j)^{1+\gamma}     \right) \\ 
 & \leq \frac{2e^{1+\gamma}}{(k+1)^2} \left( {\kappa}\|x_0-x^\star\|^2  +  \sum_{j=1}^k (j+1)^2(\delta_j-\varepsilon_j) + \frac{(j+1)^{3+\gamma}\varepsilon_j}{\gamma}\right),
\end{split}
\end{displaymath}
which yields the second part of Proposition~(\ref{prop:B}).
\end{proof}

\subsection{Proof of Proposition~\ref{prop:restart}}
Assume that for $k \geq 2$, we have the relation 
\begin{equation}
\E[h_\kmone(x_\kmone)-h_\kmone^\star] \leq \varepsilon_\kmone. \label{eq:condhk}
\end{equation}
Then, we want to evaluate the quality of the initial point $x_\kmone$ to minimize $h_k$.
\begin{equation}
\begin{split}
   h_k(x_\kmone) \!-\! h_k^\star \! & = h_\kmone(x_\kmone) + \frac{\kappa}{2}\|x_\kmone-y_\kmone\|^2 - \frac{\kappa}{2}\|x_\kmone-y_{\kmtwo}\|^2 - h_k^\star \\
                             & = h_\kmone(x_\kmone) - h_\kmone^\star + h_\kmone^\star - h_k^\star + \frac{\kappa}{2}\|x_\kmone-y_\kmone\|^2 - \frac{\kappa}{2}\|x_\kmone-y_{\kmtwo}\|^2  \\
                             & = h_\kmone(x_\kmone) - h_\kmone^\star + h_\kmone^\star - h_k^\star \!-\! \kappa(x_\kmone \!-\! y_\kmone  )^\top(y_\kmone\!-\!y_{\kmtwo})\! -\! \frac{\kappa}{2}\|y_\kmone\!-\!y_{\kmtwo}\|^2.  \\
\end{split}\label{eq:restart1}
\end{equation}
Then, we may use the fact that $h_k^\star$ can be interpreted as the Moreau-Yosida smoothing of the objective~$f$, defined as $G(y)=\min_{x \in \Real^p} F(x) + \frac{\kappa}{2}\|x-y\|^2$,
which gives us immediately a few useful results, as noted in~\cite{lin2018inexact}. Indeed, we know that $G$ is $\kappa$-smooth with $\nabla G(y_\kmone) = \kappa ( y_\kmone - x_k^\star )$ for all $k \geq 1$ and
\begin{equation}
\begin{split}
    h_\kmone^\star = G(y_{\kmtwo})  & \leq G(y_{\kmone}) + \nabla G(y_{\kmone})^\top(y_{\kmtwo}-y_{\kmone}) + \frac{\kappa}{2}\|y_\kmone - y_{\kmtwo}\|^2 \\
                                  & = h_k^\star  + \kappa (y_{\kmone} - x_k^\star)^\top(y_{\kmtwo} - y_{\kmone}) + \frac{\kappa}{2}\|y_\kmone - y_{\kmtwo}\|^2. \\
\end{split}\label{eq:restart3}
\end{equation}
Then, combining~(\ref{eq:restart1}) and~(\ref{eq:restart3}), 
\begin{equation*}
\begin{split}
   h_k(x_\kmone) - h_k^\star & \leq h_\kmone(x_\kmone) - h_\kmone^\star + \kappa ( x_\kmone - x_k^\star   )^\top (y_{\kmtwo} - y_{\kmone}). \\
                             & \leq h_\kmone(x_\kmone) - h_\kmone^\star + \kappa ( x_\kmone - x_\kmone^\star)^\top (y_{\kmtwo} - y_{\kmone}) \!+\! \kappa (x_\kmone^\star\!-\! x_k^\star)^\top (y_{\kmtwo} - y_{\kmone}) \\
                             & \leq h_\kmone(x_\kmone) - h_\kmone^\star + \kappa ( x_\kmone - x_\kmone^\star)^\top (y_{\kmtwo} - y_{\kmone}) + \kappa\|y_\kmone-y_\kmtwo\|^2 \\
                             & \leq h_\kmone(x_\kmone) - h_\kmone^\star + \frac{\kappa}{2}\|x_\kmone - x_\kmone^\star\|^2 + \frac{3\kappa}{2}\|y_\kmone-y_\kmtwo\|^2 \\
                             & \leq \frac{3}{2}(h_\kmone(x_\kmone) - h_\kmone^\star) + \frac{3\kappa}{2}\|y_\kmone-y_\kmtwo\|^2, \\
\end{split}
\end{equation*}
where the third inequality uses the non-expansiveness of the proximal operator; the fourth inequality uses the inequality $a^\top b \leq \frac{\|a\|^2}{2} + \frac{\|b\|^2}{2}$ for vectors $a,b$, and the last inequality uses the strong convexity of $h_\kmone$. 
Then, we may use the same upper-bound on $\|y_\kmone-y_\kmtwo\|$ as~\cite[Proposition 12]{catalyst_jmlr}, namely
\begin{displaymath}
   \|y_\kmone-y_\kmtwo\|^2 \leq 36 \max\left( \|x_\kmone-x^\star\|^2,\|x_\kmtwo-x^\star\|^2,\|x_{\kmthree}-x^\star\|^2\right), 
\end{displaymath}
where we define $x_{-1}=x_0$ if $k=2$.

\subsection{Proof of Proposition~\ref{prop:newcatalyst}}
The proof is similar to the derivation described in Section~\ref{appendix:sec:stoch_catalyst}.

\paragraph{Inner-loop complexity.}
With the choice of $\delta_k$, we have that $\xi_\kmone = O( \delta_\kmone/\sqrt{q})$. 
Besides, since we enforce $\E[H_k(x_k)-H_k^\star] \leq \delta_k$ for all $k \geq 0$, the result of Proposition~\ref{prop:restart} can be applied and the discussion following the proposition still applies, such that the complexity for computing $x_k$ is indeed $\tilde{O}(\tau/\eta_k)$.

\paragraph{Outer-loop complexity.}
Then, according to Proposition~\ref{prop:A}, it is easy to show that
$\E[F(x_k)-F^\star] \leq O((1-\sqrt{q}/2)^k(F(x_0)-F^\star))/\sqrt{q}$ and thus 
it suffices to choose 
\begin{displaymath}
    K = O\left(\frac{1}{\sqrt{q}}\log\left(\frac{F(x_0)-F^\star}{\sqrt{q}\varepsilon}\right)\right)
\end{displaymath}
iterations to guarantee $\E[F(x_K)-F^\star] \leq \varepsilon$.

\paragraph{Global complexity.}
We use the exact same derivations as in Section~\ref{appendix:sec:stoch_catalyst} except that 
we use the fact that $\varepsilon = O( \varepsilon_K/\sqrt{q}) = O((1-\sqrt{q}/3)^K(F(x_0)-F^\star)/\sqrt{q})$ instead
of $\varepsilon = O( \varepsilon_K/q)$, which gives us the desired complexity.

\section{Methods~$\mtd$ with Duality Gaps Based on Strongly-Convex Lower Bounds}\label{appendix:lower}
In this section, we summarize a few results from~\cite{kulunchakov2019estimate} and introduce minor modifications to guarantee the condition~(\ref{eq:newcatalyst}). 
For solving a stochastic composite objectives such as~(\ref{eq:risk}), where $F$ is $\mu$-strongly convex, consider an algorithm~$\mtd$ performing the following classical updates 
\begin{displaymath}
   z_t \leftarrow \text{Prox}_{\eta \psi}\left[z_{t-1} - \eta g_t    \right]~~~~\text{with}~~~~ \E[g_t|\Fcal_\kmone] = \nabla f(z_{t-1}),
\end{displaymath}
where $\eta \leq 1/L$, and the variance of $g_t$ is upper-bounded by $\sigma_t^2$. Inspired by estimate sequences from~\cite{nesterov}, the authors of~\cite{kulunchakov2019estimate} build recursively a
$\mu$-strongly convex quadratic function $d_t$ of the form 
\begin{displaymath}
    d_t(z) = d_t^\star + \frac{\mu}{2}\|z_t-z\|^2.
\end{displaymath}
From the proof of Proposition 1 in~\cite{kulunchakov2019estimate}, we then have 
\begin{displaymath}
   \E[d_t^\star] \geq (1-\eta\mu)\E[d_\kmone^\star] + \eta\mu \E[F(z_t)] - \eta^2 \mu \sigma_t^2,
\end{displaymath}
which leads to 
\begin{displaymath}
   F^\star - \E[d_t^\star] +  \eta\mu (\E[F(z_t)]-F^\star) \leq (1-\eta\mu)\E[F^\star - d_\kmone^\star]+ \eta^2 \mu \sigma_t^2,
\end{displaymath}
which is a minor modification of Proposition 1 in~\cite{kulunchakov2019estimate} that is better suited to our purpose.

\paragraph{With constant variance.}
Assume now that $\sigma_t = \sigma$ for all $k \geq 1$.
Following the iterate averaging procedure used in Theorem 1 of~\cite{kulunchakov2019estimate}, which produces an iterate $\hat{z}_t$, we obtain
\begin{equation}
\begin{split}
   \E[F(\hat{z}_t) - d_t^\star]  & \leq \left( 1- \eta\mu \right)^t \left( F(z_0)- d_0^\star \right) + \eta\sigma^2, \\
\end{split}\label{eq:resD}
\end{equation}
where $d_0^\star$ can be freely specified for the analysis: it is not used by the algorithm, but it influences~$d_t^\star$ through the relation $\E[d_t(z)] \leq \Gamma_t d_0(z) + (1-\Gamma_t)\E[F(z)]$ with $\Gamma_t = (1-\mu\eta)^k$, see Eq.~(11) in~\cite{kulunchakov2019estimate}. 
In contrast, Theorem 1 in~\cite{kulunchakov2019estimate} would give here
\begin{equation}
\begin{split}
   \E[F(\hat{z}_t) - F^\star + d_t(z^\star) - d_t^\star]  & \leq \left( 1- \eta\mu \right)^t \left( 2(F(z_0)- F^\star) \right)  + \eta \sigma^2, \\
\end{split}\label{eq:thm2}
\end{equation}
where $z^\star$ is a minimizer of $F$, which is sufficient to guarantee~(\ref{eq:aux}) given that $d_t(z^\star) \geq d_t^\star$.

\paragraph{Application to the minimization of $H_k$.}
Let us now consider applying the method to an auxiliary function $H_k$ from~(\ref{eq:newcatalyst}) instead of $F$, with initialization $x_\kmone$. After running $T$ iterations, 
define $h_k$ to be the corresponding function $d_T$ defined above and $x_k = \hat{z}_T$. $H_k$ is $(\kappa+\mu)$-strongly convex and thus $h_k$ is also $(\kappa+\mu)$-strongly convex such that \HypA~is satisfied. Let us now check possible choices for $d_0^\star$ to ensure~\HypB. For $z_\kmone=\alpha_\kmone x^\star + (1-\alpha_\kmone)x_\kmone$,
$
\E[d_T(z_\kmone)] \leq \Gamma_T d_0(z_\kmone) + (1-\Gamma_T) H_k(z_\kmone)
$ such that we simply need to choose $d_0^\star$ such that $\E[d_0(z_\kmone)] \leq \E[H_k(z_\kmone)]$. 
Then, choose
\begin{equation}
 d_0^\star  = H_k^\star - F(x_\kmone)+F^\star, \label{eq:d0}
\end{equation}
and
\begin{displaymath}
\begin{split}
   d_0(z_\kmone) & = d_0^\star + \frac{\kappa+\mu}{2}\|x_\kmone - z_\kmone\|^2  \\
          & = d_0^\star + \frac{(\kappa+\mu) \alpha_\kmone^2}{2}\|x_\kmone - x^\star\|^2  \\
          & = d_0^\star + \frac{\mu}{2}\|x_\kmone - x^\star\|^2  \\
          & \leq  d_0^\star + F(x_\kmone)-F^\star  
           =  H_k^\star    
           \leq  H_k(z_\kmone),   \\
\end{split}
\end{displaymath}
such that \HypB~is satisfied, and finally~(\ref{eq:resD}) becomes
\begin{equation*}
\E[H_k(x_k) - h_k^\star] \leq \left( 1- \eta(\mu+\kappa) \right)^T \left( H_k(x_\kmone)-H_k^\star + F(x_\kmone)-F^\star\right) + \eta\sigma^2,
\end{equation*}
which matches~(\ref{eq:newcatalyst}).

\paragraph{Variance-reduction methods.}
In~\cite{kulunchakov2019estimate}, gradient estimators $g_t$ with variance
reduction are studied, leading to variants of SAGA~\cite{saga},
MISO~\cite{miso}, and SVRG~\cite{proxsvrg}, which can deal with the stochastic
finite-sum problem presented in Section~\ref{sec:intro}. Then, the variance of
$\sigma_t^2$ decreases (Proposition 2 in~\cite{kulunchakov2019estimate}).

Let us then consider again the guarantees of the method obtained when minimizing $F$ with $\frac{\mu}{L} \leq \frac{1}{5n}$.
From Corollary 5 of~\cite{kulunchakov2019estimate}, we have 
\begin{equation*}
\E[F(\hat{z}_t) - F^\star + d_t(z^\star)-d_t^\star] \leq 8\left( 1-\mu \eta \right)^t\left(F(x_0)- F^\star  \right)  + 18\eta\sigma^2, 
\end{equation*}
and~(\ref{eq:aux}) is satisfied.
Consider now two cases at iteration $T$:
\begin{itemize}[leftmargin=*]
  \item if $\E[d_T(z^\star)] \geq F^\star$, then we have $\E[F(\hat{z}_T) -d_T^\star] \leq 8\left( 1-\mu \eta \right)^T\left(F(x_0)- F^\star  \right)  + 18\eta\sigma^2$.
  \item otherwise, it is easy to modify Theorem 2 and Corollary 5 of~\cite{kulunchakov2019estimate} to obtain 
\begin{equation}
\E[F(\hat{z}_T) - d_T^\star] \leq \left( 1-\mu \eta \right)^T\left(2(F(x_0)- F^\star) + 6 (F^\star - d_0^\star)  \right)  + 18\eta\sigma^2,\label{eq:daux}
\end{equation}
\end{itemize}

\paragraph{Application to the minimization of $H_k$.}
Consider now applying the method for minimizing $H_k$, with the same choice of
$d_0^\star$ as~(\ref{eq:d0}), which ensures~\HypB, and same definitions as above for $x_k$ and $h_k$.
Note that the conditions on $\mu$ and $L$ above are satisfied when $\kappa=\frac{L}{5n}-\mu$ under the condition $\frac{L}{5n} \geq \mu$.
Then, we have from the previous results, after replacing $F$ by $H_k$ making the right subsitutions
\begin{displaymath}
\E[H_k(x_k) - h_k^\star] \leq \left( 1-(\mu+\kappa) \eta \right)^T\left(8(H_k(x_\kmone)- H_k^\star) + 6 (F(x_\kmone) - F^\star)  \right)  + 18\eta\sigma^2,\label{eq:daux}
\end{displaymath}
and~(\ref{eq:newcatalyst}) is satisfied.

\paragraph{Other schemes.}
Whereas we have presented approaches were $d_t$ is
quadratic,~\cite{kulunchakov2019estimate} also studies another class of
algorithms where $d_t$ is composite (see Section 2.2
in~\cite{kulunchakov2019estimate}). The results we present in this paper can be
extended to such cases, but for simplicity, we have focused on quadratic surrogates.

\section{Additional Experimental Material}\label{appendix:exp}

\paragraph{Computing resources.}
The numerical evaluation was performed by using four nodes of a CPU cluster with
56 cores of Intel CPUs each. The full set of experiments presented in this
paper (with $5$ runs for each setup) takes approximately half a day.

\paragraph{Making plots.}
We run each experiment five times and average the outputs. 
We display plots on a logarithmic scale for the primal gap~$F(x_k)-F^\star$
(with~$F^\star$ estimated as the minimum value observed from all runs). Note that for SVRG, one iteration is considered to perform two epochs since it requires accessing the full dataset every $n$ iterations on average. 

\subsection{Additional experiments.}

\paragraph{Acceleration with no noise, $\delta=0$.}
We start evaluating the acceleration approach when there is no noise. This is essentially evaluating the original Catalyst method~\cite{catalyst_jmlr} in a deterministic setup in order to obtain a baseline comparison when $\delta=0$. The results are presented in Figures~\ref{1dropout0loss0.pdf} and~\ref{2dropout0loss0.pdf} for the logistic regression problem.
As predicted by theory, acceleration is more important when conditioning is low (bottom curves).

\begin{figure}[!hptb]
  \includegraphics[width=\linewidth]{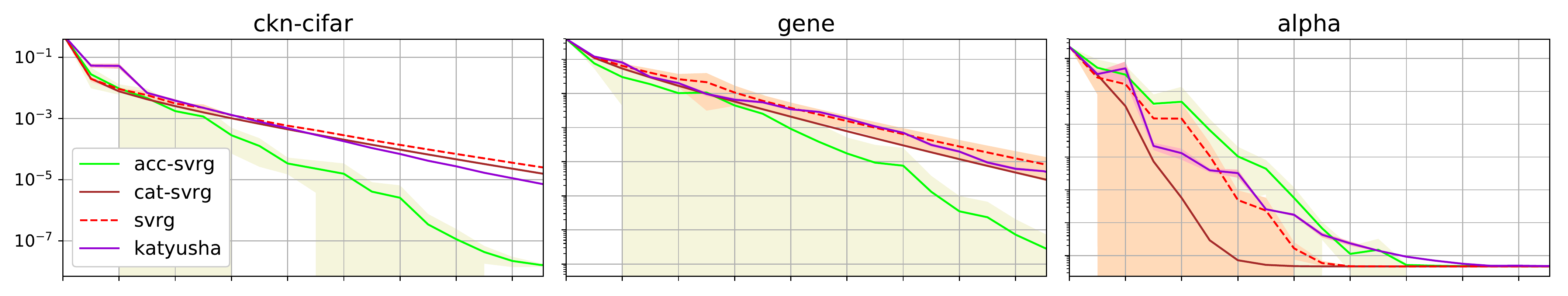}\vspace{-0.22cm}\\
  \includegraphics[width=\linewidth]{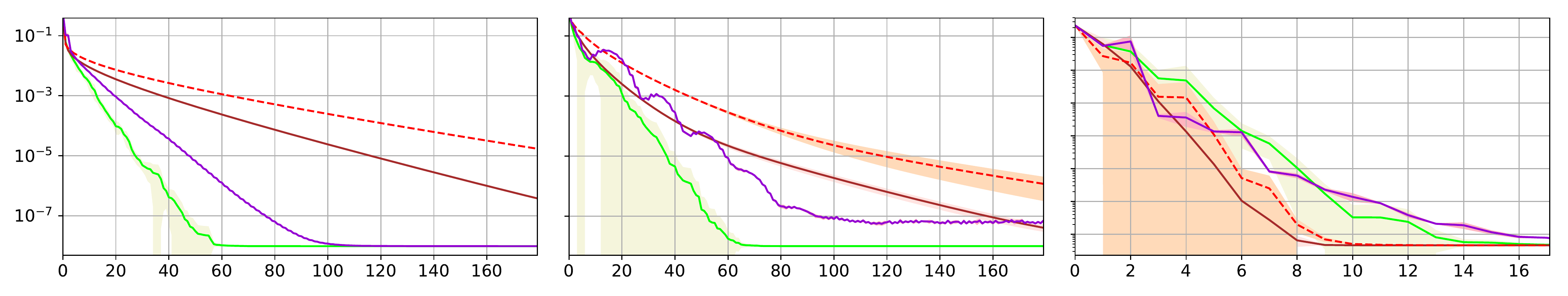}\\
  \vspace*{-0.5cm}
  \caption{Accelerating SVRG-like methods for $\ell_2$-logistic regression with $\mu=1/(10 n)$ (top) and $\mu = 1/(100 n)$  (bottom) for $\delta=0$.
  All plots are on a logarithmic scale for the objective function value, and the $x$-axis denotes the number of epochs.
  The colored tubes around each curve denote a standard deviations across $5$ runs. They do not look symmetric because of the logarithmic scale.
  }\label{1dropout0loss0.pdf}
\end{figure}

\begin{figure}[!hptb]
  \includegraphics[width=\linewidth]{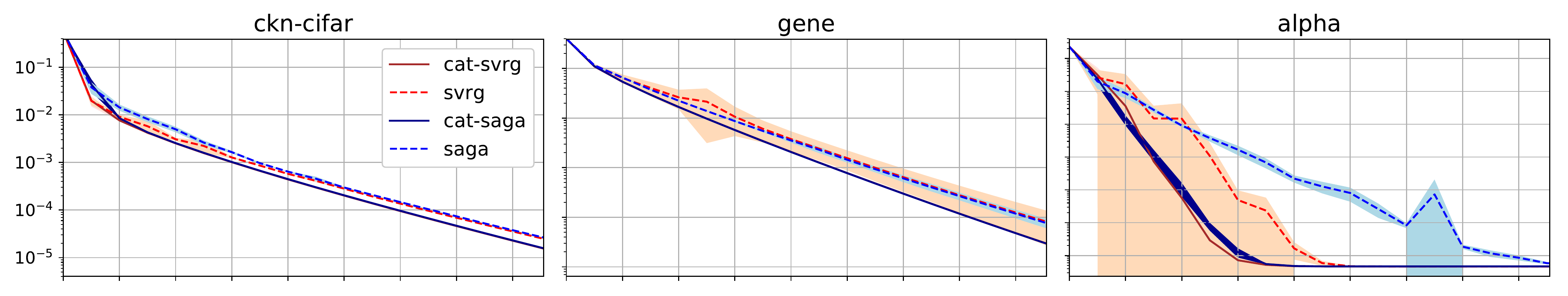}\vspace{-0.22cm}\\
  \includegraphics[width=\linewidth]{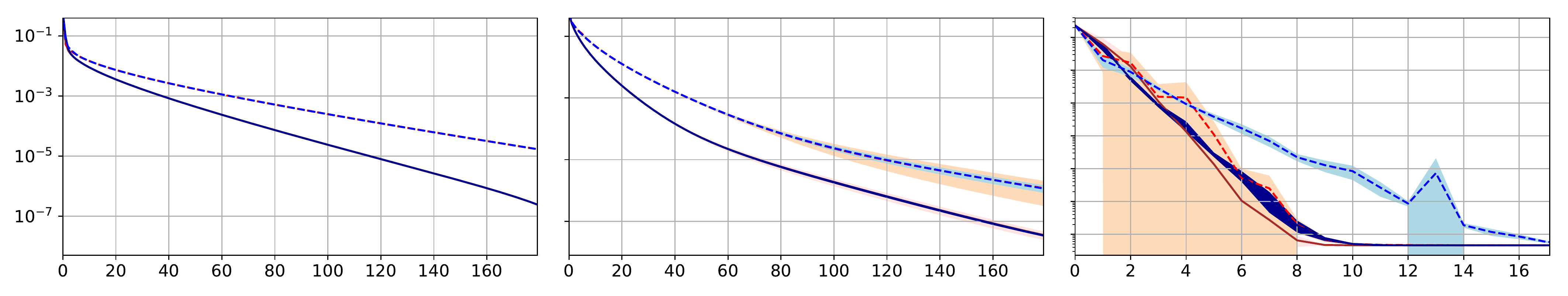}\\
  \vspace*{-0.5cm}
  \caption{Same plots as in Figure~\ref{1dropout0loss0.pdf} when comparing SVRG and SAGA, with no noise ($\delta=0$) with $\mu=1/(10 n)$ (top) and $\mu = 1/(100 n)$  (bottom) .}\label{2dropout0loss0.pdf}
\end{figure}

\paragraph{Stochastic acceleration with no noise, $\delta=0.01$ and $\delta=0.1$.}
Then, we perform a similar experiments by adding noise and report the results in Figures~\ref{1dropout1loss0.pdf}, \ref{2dropout1loss0.pdf}, \ref{1dropout2loss0.pdf}, \ref{2dropout2loss0.pdf}.
In general, the stochastic Catalyst approach seems to perform on par with the accelerated SVRG approach of~\cite{kulunchakov2019estimate} and even better in one case.

\begin{figure}[!htbp]
  \includegraphics[width=\linewidth]{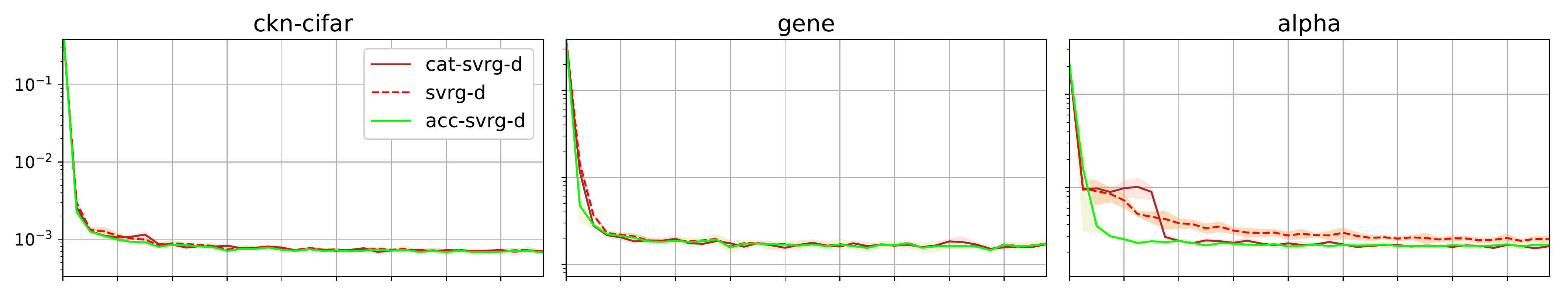}\vspace{-0.22cm}\\
  \includegraphics[width=\linewidth]{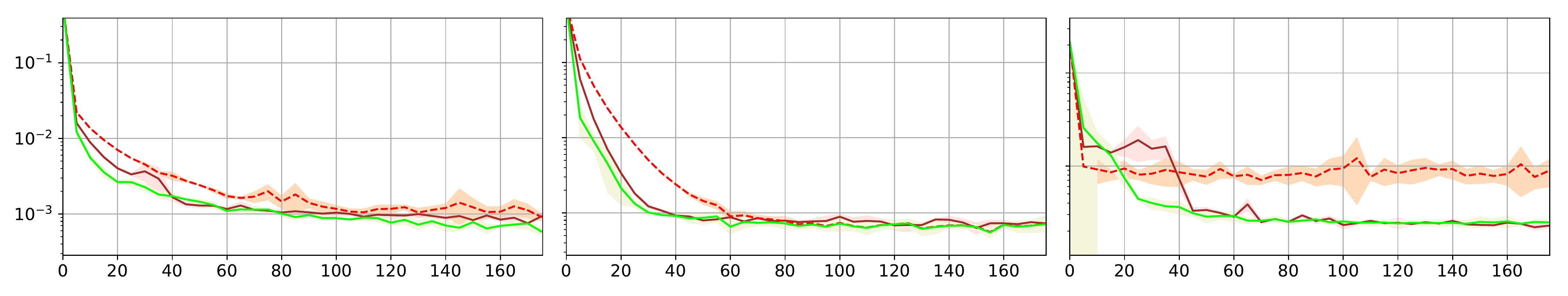}\\
  \vspace*{-0.5cm}
  \caption{Same plots as in Figure~\ref{1dropout0loss0.pdf} for $\delta=0.01$ with $\mu=1/(10 n)$ (top) and $\mu = 1/(100 n)$  (bottom).}\label{1dropout1loss0.pdf}
\end{figure}

\begin{figure}[!htb]
  \includegraphics[width=\linewidth]{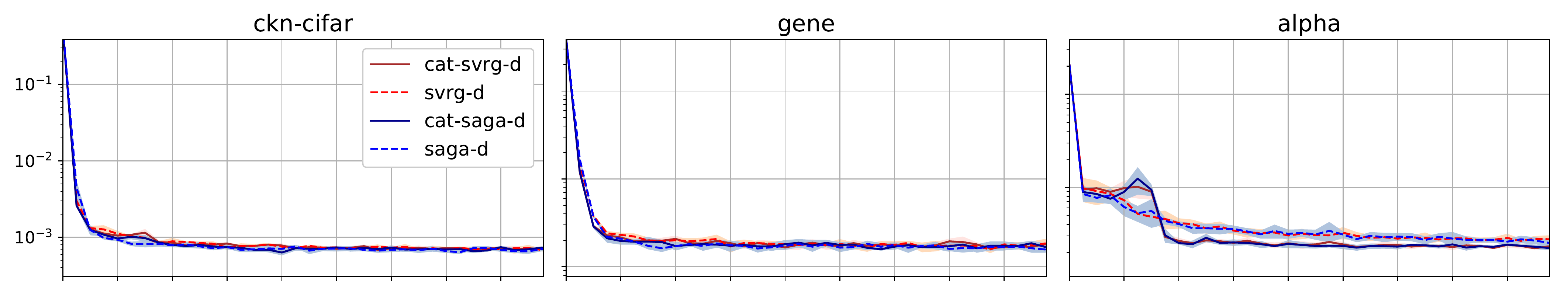}\vspace{-0.22cm}\\
  \includegraphics[width=\linewidth]{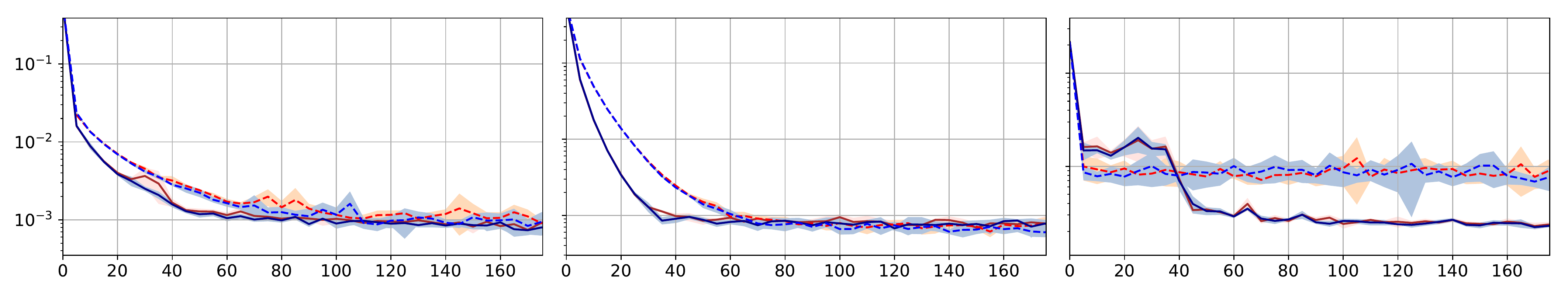}\\
  \vspace*{-0.5cm}
  \caption{Same plots as in Figure~\ref{2dropout0loss0.pdf} for $\delta=0.01$ with $\mu=1/(10 n)$ (top) and $\mu = 1/(100 n)$  (bottom).}\label{2dropout1loss0.pdf}
\end{figure}

\begin{figure}[!htb]
  \includegraphics[width=\linewidth]{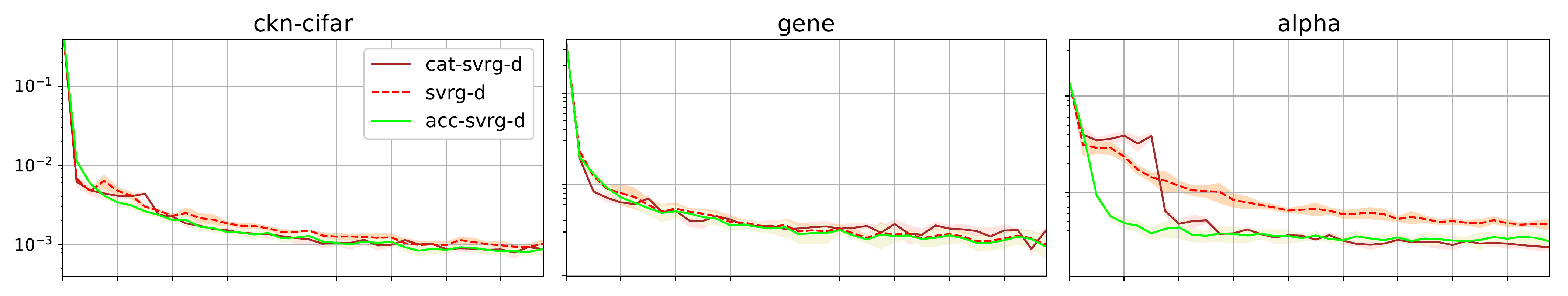}\vspace{-0.22cm}\\
  \includegraphics[width=\linewidth]{one-pass-iv_lam100_d10_loss0.pdf}\\
  \vspace*{-0.5cm}
  \caption{Same plots as in Figure~\ref{1dropout0loss0.pdf} for $\delta=0.1$ with $\mu=1/(10 n)$ (top) and $\mu = 1/(100 n)$  (bottom).}\label{1dropout2loss0.pdf}
\end{figure}

\begin{figure}[!htb]
  \includegraphics[width=\linewidth]{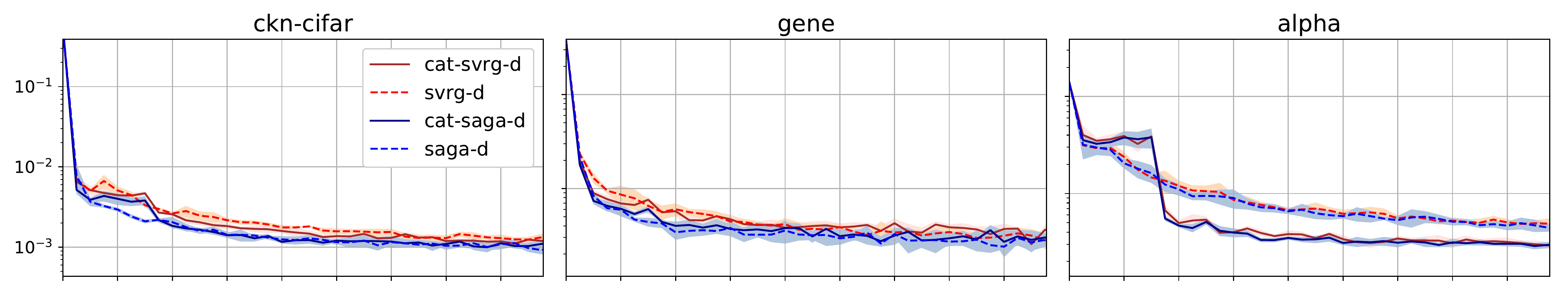}\vspace{-0.22cm}\\
  \includegraphics[width=\linewidth]{one-pass-vi_lam100_d10_loss0.pdf}\\
  \vspace*{-0.5cm}
  \caption{Same plots as in Figure~\ref{2dropout0loss0.pdf} for $\delta=0.1$ with $\mu=1/(10 n)$ (top) and $\mu = 1/(100 n)$  (bottom).}\label{2dropout2loss0.pdf}
\end{figure}

\paragraph{Evaluating the square hinge loss.}
In Figure~\ref{1loss1.pdf}, we perform experiments using the square hinge loss, where the methods perform similarly as for the logistic regression case, despite the fact that the bounded noise assumption does not necessarily hold on the optimization domain for the square hinge loss.

\begin{figure}[!htb]
  \includegraphics[width=\linewidth]{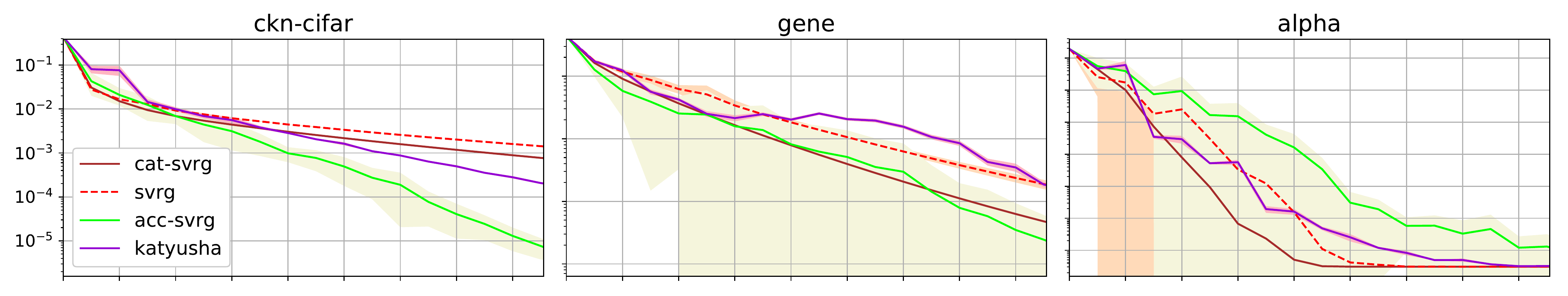}\vspace{-0.2cm}\\
  \includegraphics[width=\linewidth]{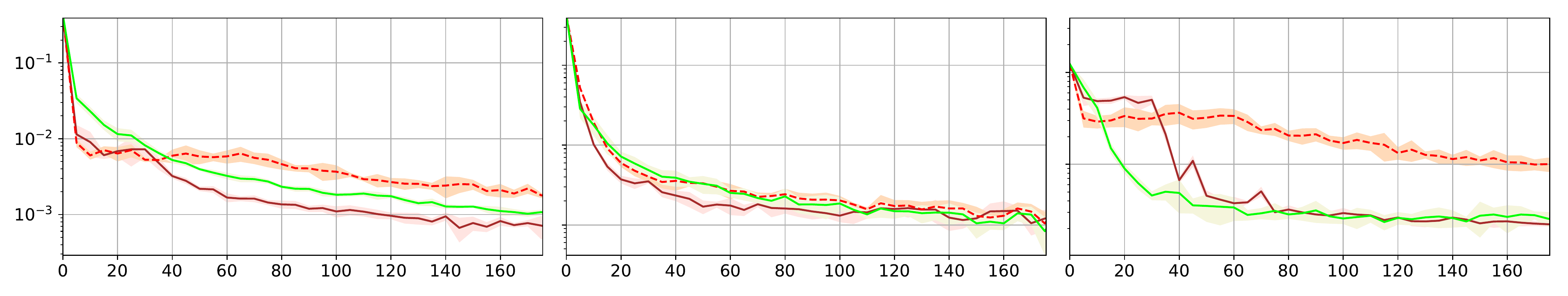}\\
  \vspace*{-0.5cm}
  \caption{Accelerating SVRG-like methods when using the squared hinge loss instead of the logistic for $\delta=0$ (top) and $\delta=0.1$, both with $\mu=1/(10 n)$.}\label{1loss1.pdf}
\end{figure}

\begin{figure}[!htb]
  \includegraphics[width=\linewidth]{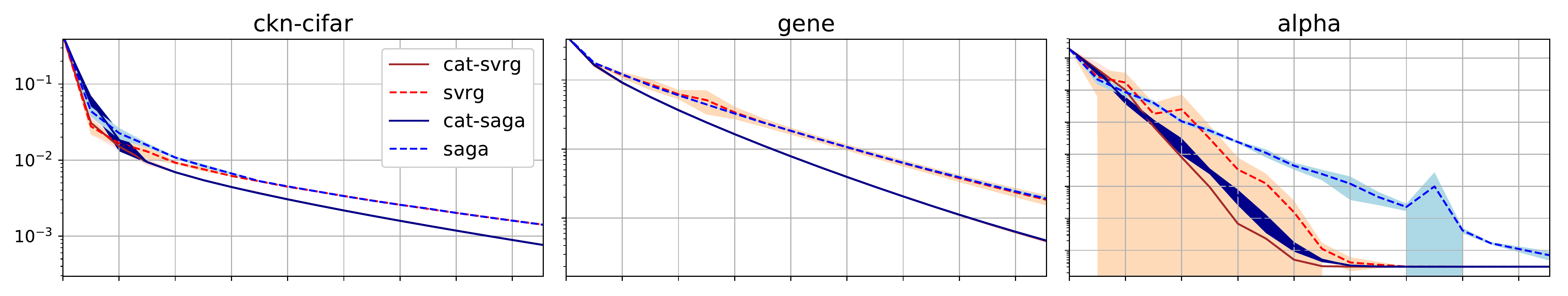}\vspace{-0.2cm}\\
  \includegraphics[width=\linewidth]{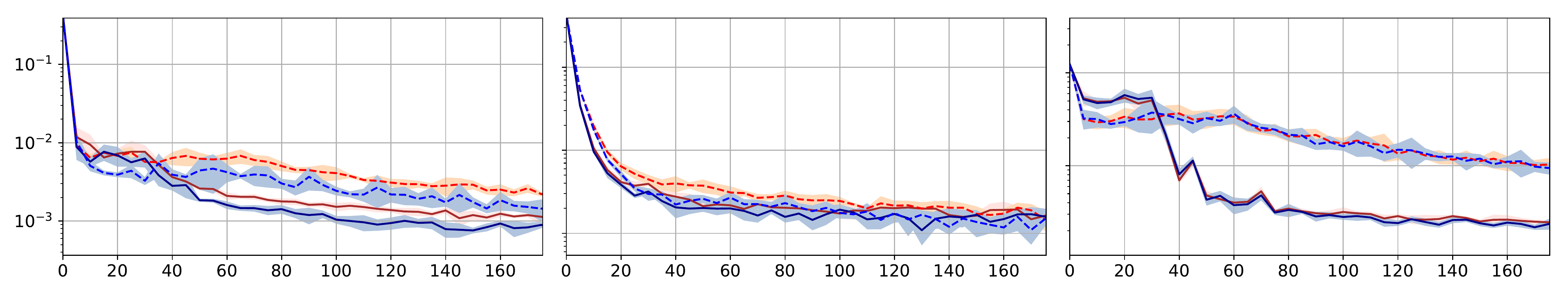}\\
  \vspace*{-0.5cm}
  \caption{Same plots as in Figure~\ref{1loss1.pdf} for SVRG and SAGA, with $\delta=0$ (top) and $\delta=0.1$ for $\mu=1/(10 n)$.}\label{2loss1.pdf}
\end{figure}

\paragraph{Evaluating ill-conditioned problems.}
Finally, we study in Figure~\ref{instabilities} how the methods behave when the problems are badly conditioned. There, acceleration seems to work on ckn-cifar, but fails on gene and alpha, suggestions that acceleration is difficult to achieve when the condition number is extremely low.
\begin{figure}[!htb]
  \includegraphics[width=\linewidth]{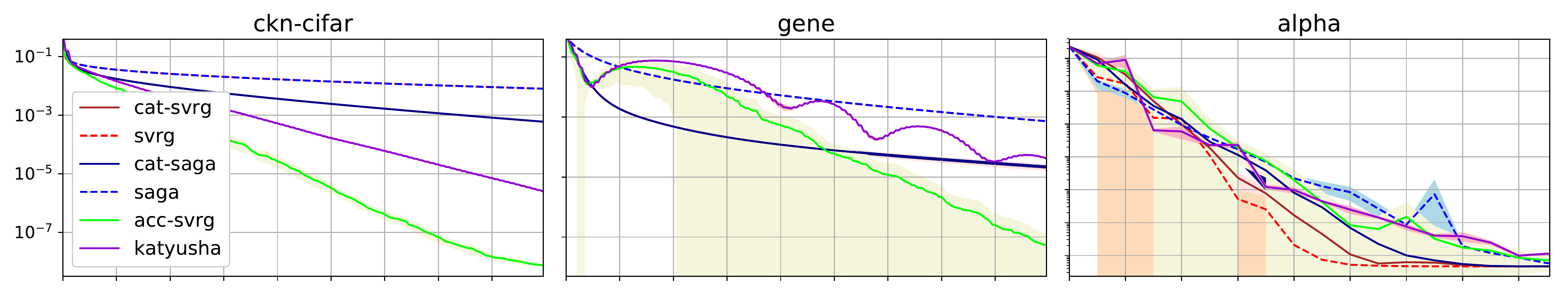}\vspace{-0.2cm}\\
  \includegraphics[width=\linewidth]{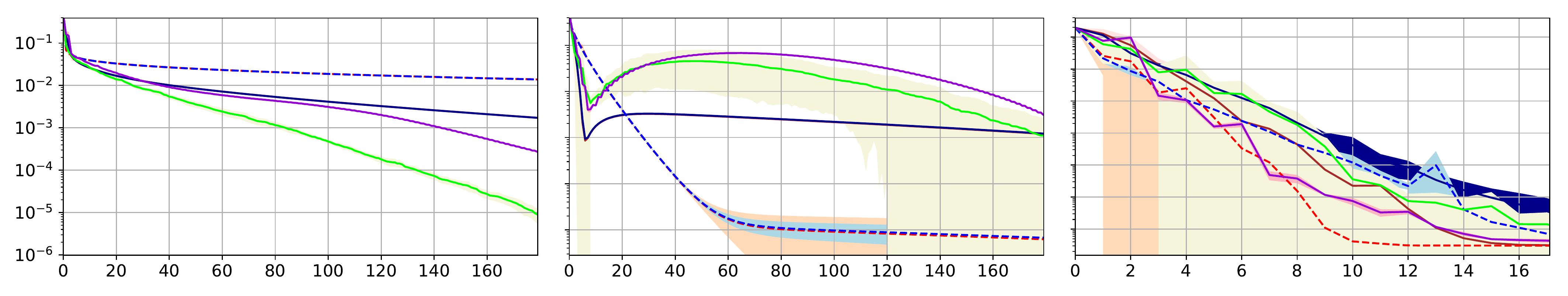}\\
  \caption{Illustration of potential numerical instabilities problems when the problem is very ill-conditioned. We use~$\mu=1/(1000 n)$ with $\delta=0$ for the logistic loss (top) and squared hinge (bottom).}\label{instabilities}
\end{figure}

\end{document}